\documentclass[a4paper,draft]{article}
\usepackage{amssymb,amsmath,amsthm}

\theoremstyle{plain}
\newtheorem{theorem}{\sc Theorem}[section]
\newtheorem{lemma}{\sc Lemma}[section]
\newtheorem{proposition}{\sc Proposition}[section]

\theoremstyle{definition}
\newtheorem{example}{\sc Example}[section]

\newcommand{\calB}{\mathcal{B}}
\newcommand{\Nd}{\mathbb{N}}
\newcommand{\Qd}{\mathbb{Q}}
\newcommand{\Rd}{\mathbb{R}}
\newcommand{\Zd}{\mathbb{Z}}
\newcommand{\eps}{\varepsilon}
\newcommand{\e}{\mathrm{e}}
\newcommand{\Prob}{\mathsf{P}}
\newcommand{\Mean}{\mathsf{E}}
\newcommand{\maps}{\mathrel{\colon}} 
\newcommand{\tends}[1]{\xrightarrow[#1]{}}
\let\0=\varnothing
\newcommand{\1}{\mathsf{1}}

\newcommand{\Ceiling}[1]{\left\lceil #1 \right\rceil}
\newcommand{\floor}[1]{\lfloor #1 \rfloor}
\newcommand{\Floor}[1]{\left\lfloor #1 \right\rfloor}
\newcommand{\tsum}{{\textstyle{\sum}}}

\newcommand{\rint}{\mathrm{int}}
\newcommand{\clos}[1]{\overline{#1}}                      
\renewcommand{\le}{\leqslant}
\renewcommand{\ge}{\geqslant}
 
\newcommand{\conv}{\mathrm{conv}}
\newcommand{\cone}{\mathrm{cone}}
\newcommand{\abs}[1]{\lvert #1\rvert} 
\newcommand{\norm}[1]{\lVert #1\rVert} 
\newcommand{\be}{\setminus}
\newcommand{\cups}{\cup\cdots\cup}

\newcommand\refeq[1]{{\rm (\ref{e:#1})}}

\usepackage{natbib}
\newcommand{\aff}{\mathrm{aff}}
\newcommand{\lin}{\mathrm{lin}}
\let\Lin=\lin
\newcommand{\gp}{\mathrm{gp}}
\newcommand{\sg}{\mathrm{sg}}
\newcommand{\verts}{\mathrm{vert}}

\title{An ergodic theorem for subadditive random functions on vector semigroups}
\author{Vytautas Kazakevi\v cius\footnote{vytautas.kazakevicius@mif.vu.lt}\\Vilnius university}

\begin{document}
\maketitle

\begin{abstract}
Let $f=(f^x\mid x\in S)$, $S\subset\Zd^m$, be a semigroup of ergodic measure-preserving transformations of a probability space $(\Omega,\Prob)$ and $h$ a real random function on $S$, such that $h(x+y,\omega)\le h(x,\omega)+h(y,f^x\omega)$ for all $x,y\in S$ and $\omega\in\Omega$. We prove that there exists a sublinear function $q\maps O\to[-\infty;\infty)$ defined on $O=\rint(\cone(S))$, and a set $W\subset\Omega$ of full probability, such that $h(x_n,\omega)/\abs{x_n}\to q(x)$ for all $\omega\in W$ and all sequences $(x_n)\subset S$ with asymptotic direction $x\in O$. The moment condition for this reflects the size of the semigroup $f$, not that of $S$. However, an additional independence assumption about $h$ is made.

\emph{Keywords:} vector semigroup, gauge, subadditive ergodic theorem, first passage percolation.
\end{abstract}

\section{Introduction}

The main theorem of this paper is inspired by two results in two seemingly unrelated areas of probability: the Cox-Durrett shape theorem well known in the theory of first passage percolation, and our recent subadditive ergodic theorem for double sequences which proved useful in studying IARCH processes, well known in econometrics. Let us describe these results in more detail.

Consider the graph $G=(S,E)$ with the set of vertices $S=\Zd^2$ and the set of edges $E=\{\{x,y\}\mid \abs{x-y}=1\}$, where $\abs{x}=\abs{x_1}+\abs{x_2}$ for $x=(x_1,x_2)\in S$. Let $(\eps_e\mid e\in E)$ be a family of independent copies of some non-negative random variable $\eps$ and, for all $x,y\in S$,
\begin{equation*}
  \Delta(x,y)=\inf_{\gamma\in\Gamma(x,y)}\tau_{\gamma},\quad \tau_{\gamma}=\sum_{i=1}^n\eps_{\{x_{i-1},x_i\}} \text{ for $\gamma=(x_0,\dots,x_n)$},
\end{equation*}
where $\Gamma(x,y)$ is the set of all paths from $x$ to $y$ in $G$. The value of $\eps_e$ is interpreted as the time needed for, say, water to percolate through edge $e$. Then $\Delta(x,y)$ is the first moment when the water appears at $y$ if the percolation started at $x$. It is easily shown that $\Delta$ is a random semi-metric on $S$.

Denote $h(x)=\Delta(0,x)$. We may assume that all random variables are defined on the sample probability space $(\Omega,\Prob)=(\Rd_+^E,\lambda^E)$, where $\lambda$ is the distribution of $\eps$, and $\eps_e(\omega)=\omega_e$ for $\omega=(\omega_e\mid e\in E)\in\Omega$. Then the triangle inequality for $\Delta$ yields the following subadditivity property of $h$:
\begin{equation}\label{e:h}
  h(x+y,\omega)\le h(x,\omega)+h(y,f^x\omega),
\end{equation}
where $f^x\omega=(\omega_{x+e}\mid e\in E)$ and $x+e=\{x+y,x+z\}$ for $e=\{y,z\}\in E$. It is easily seen that each $f^x$ is a measure-preserving transformation of $\Omega$ and, for all $x,y\in S$,
\begin{equation}\label{e:f}
  f^{x+y}=f^xf^y,
\end{equation}
where the product of two transformations is understood as their composition.

Under appropriate conditions, the Kingman's subadditive ergodic theorem \citep{Kingman-68,Kingman-73} yields the existence of a norm $q$ on $\Rd^2$, such that for all $x\in S$ almost surely
\begin{equation}\label{e:q}
  \frac{h(nx)}{n}\tends{n\to\infty}q(x).
\end{equation}
The Cox-Durrett shape theorem \citep[Theorem~3]{Cox-81} can be equivalently reformulated as follows: there exists a $W\subset\Omega$ with $\Prob(W)=1$, such that
\begin{equation}\label{e:shape}
  \frac{h(x_n,\omega)}{\abs{x_n}}\to q(x)
\end{equation}
for all $\omega\in W$ and all $(x_n)\subset S$ with
\begin{equation}\label{e:adir}
  \abs{x_n}\to\infty\quad\text{and}\quad\frac{x_n}{\abs{x_n}}\to x.
\end{equation}
If \refeq{adir} holds, we call $x$ the \emph{asymptotic direction} of $(x_n)$.

In econometrics, an IARCH process $(X_t\mid t\in\Zd)$ is defined as a stationary solution to the system of equations
\begin{equation*}
  X_t=\Bigl(a_0+\sum_{i\ge 1}a_iX_{t-i}\Bigr)\eps_t,\quad t\in\Zd,
\end{equation*}
where $a_0>0$, $(a_i\mid i\ge 1)$ is a sequence of nonnegative numbers with $\sum_{i\ge 1}a_i=1$ and $(\eps_t\mid t\in\Zd)$ is a family of independent copies of a nonnegative random variable $\eps$ with $\Mean\eps=1$. It is known that such a process exists if and only if almost surely
\begin{equation}\label{e:sum}
  \sum_{\substack{k\ge 1\\i_1,\dots,i_k\ge 1}}a_{i_1}\cdots a_{i_k}\eps_{i_1}\eps_{i_1+i_2}\cdots\eps_{i_1+\cdots+i_k}<\infty.
\end{equation}
However, this condition is not easily checked in practice. The state of the art technique \citep{Kazak-2018} consists in writing this sum as $\sum_{k,n}\eta_{k,n}$, where
\begin{equation*}
  \eta_{k,n}=\sum_{i_1+\cdots+i_k=n}a_{i_1}\cdots a_{i_k}\eps_{i_1}\cdots\eps_{i_1+\cdots+i_k},
\end{equation*}
and using the subadditive ergodic theorem for double sequences \citep{Kazak-2018b} to get an exponential upper bound for the main part of the sum.

Again, we can assume that all random variables in \refeq{sum} are defined on the sample probability space $(\Omega,\Prob)=(\Rd_+^\Nd,\lambda^{\Nd})$, where $\lambda$ is the distribution of $\eps$, and $\eps_i(\omega)=\omega_i$ for $\omega=(\omega_i)\in\Omega$. Denote
\begin{equation*}
  S=\{(k,n)\in\Nd^2\mid\exists i_1,\dots,i_k\ge 1\ (a_{i_1}\cdots a_{i_k}>0,\ i_1+\cdots+i_k=n)\}
\end{equation*}
and, for $(k,n)\in S$, $h(k,n)=-\log\eta_{k,n}$. Then $S$ is an additive semigroup and $h$ a random function on it which satisfies \refeq{h} with measure-preserving transformations $f^x$ defined by $f^{(k,n)}\omega=(\omega_{i+n}\mid i\ge 1)$. Under appropriate conditions, applying the Kingman's subadditive ergodic theorem yields \refeq{q} with some function $q\maps S\to[-\infty;\infty)$, which can be extended in a unique way to a function defined on $O=\rint(\cone(S))$ (the interior of the convex cone generated by $S$; we call $O$ the \emph{asymptotic cone of $S$}). The extended function is a \emph{gauge}, i.e.\ a convex and positively homogeneous function on $O$. Theorem~2 of \citet{Kazak-2018b} can be reformulated as follows: there exists a $W\subset\Omega$ with $\Prob(W)=1$, such that \refeq{shape}~holds for all $\omega\in W$ and all $(x_n)\subset S$ with asymptotic direction $x\in O$.

A natural question arises if there can be proved a general theorem including both the Cox-Durrett shape theorem and Theorem~2 of \citet{Kazak-2018b} as special cases. Let us think what such a theorem would look like. We are given some additive semigroup $S\subset \Zd^m$ (let us call it a \emph{vector semigroup}), a probability space $(\Omega,\Prob)$, and a family $f=(f^x\mid x\in S)$ of measure-preserving transformations satisfying \refeq{f}. If $S$ contains 0, we additionally assume that $f^0$ is the identity transformation. We call $f$ \emph{an action} of $S$ on the probability space $(\Omega,\Prob)$. We call an action $f$ \emph{ergodic} if, for all $x\in S\be\{0\}$, the probability $\Prob$ is $f^x$-ergodic, i.e.\ $\Prob(W)\in\{0,1\}$ for any measurable $W\subset\Omega$ with $(f^x)^{-1}(W)=W$.

Next, we are given some family $h=(h(x)\mid x\in S)$ of random variables. We write $h(x,\omega)$ for the value of $h(x)$ at the point $\omega\in\Omega$ and think of $h$ as of a random function from $S$ to $\Rd$. We call that function \emph{$f$-subadditive} if \refeq{h} holds for all $\omega\in\Omega$ and $x,y\in S$. If $h$ is $f$-subadditive then, for all $x\in S$, the sequence $(h(nx)\mid n\ge 1)$ is subadditive in the usual sense and the Kingman's subadditive ergodic theorem can be applied. It yields (provided $\Mean h^+(x)<\infty$) the existence of $q(x)\in[-\infty;\infty)$ such that almost surely \refeq{q} holds. In general, $q(x)$ is a random variable (and so $q$ is a random function on $S$). However, it is almost $f^x$-invariant and therefore almost surely equals some constant if the probability $\Prob$ is $f^x$-ergodic. Hence, if the action $f$ ir ergodic, we may assume that $q(x)$ is non-random for all $x\ne 0$. We want $q(0)$ to be non-random as well, so we supplement the definition of $f$-subadditivity by the requirement $h(0)=0$, in case, where $0\in S$.

It follows from \refeq{q} that, for all $x,y\in S$ and $t\in\Nd$,
\begin{equation*}
  q(x+y)\le q(x)+q(y)\quad\text{and}\quad q(tx)=tq(x),
\end{equation*}
so we may call $q$ a \emph{$\Zd$-gauge} on $S$. It can be shown that every $\Zd$-gauge is extended in a unique way to a gauge on the asymptotic cone of the semigroup $S$. We denote the extended function by the same letter $q$ and call it \emph{the gauge associated with $h$}.

The Cox-Durrett shape theorem is valid if $\Mean h(x)^2<\infty$ for all $x\in S=\Zd^2$. The corresponding assumption in Theorem~2 of \citet{Kazak-2018b} is $\Mean h^+(x)<\infty$. So the moment condition reflects the size of the semigroup $\{f^x\mid x\in S\}$ of measure-preserving transformations and not that of the semigroup $S$, which in both cases is two-dimensional (we define the dimension of the semigroup $S$ as $\dim\lin(S)$, the dimension of the linear subspace generated by $S$). To take account of this, we should assume that the action $f$ is of the form
\begin{equation}\label{e:g}
  f^x=g^{\pi(x)},
\end{equation}
where $g$ is an action of some $l$-dimensional semigroup $T$ and $\pi$ is a semigroup homomorphism from $S$ to $T$, i.e.\ $\pi(x+y)=\pi(x)+\pi(y)$ for all $x,y\in S$. We call $\pi$ \emph{nontrivial} if $\pi(x)\ne 0$ for some $x\ne 0$. For example, in Theorem~2 of \citet{Kazak-2018b} \refeq{g} holds with $T=\Nd$, $\pi(k,n)=n$ and $g^n$ defined by $g^n\omega=(\omega_{i+n}\mid i\ge 1)$. The theorem we are targeting should state that if $\Mean (h^+(x))^l<\infty$ for all $x$ then \refeq{shape} holds for any $\omega$ in a set $W$ of full probability and all $(x_n)\subset S$ with asymptotic direction $x\in O$.
Unfortunately, we managed to prove it only in the case, where some additional rather strong condition is satisfied.

We call a family of random variables $(Z_y\mid y\in T)$ \emph{almost independent} if there exists a $c<\infty$, such that, for all $A,B\subset T$ with $\rho(A,B)>c$, the subfamilies $(Z_y\mid y\in A)$ and $(Z_y\mid y\in B)$ are independent. Here $\rho(A,B)=\inf_{x\in A,y\in B}\abs{x-y}$ is the usual distance between sets $A$ and $B$ in the space $\Rd^n$ comprising $T$. The main result of the paper is the following theorem.

\begin{theorem}\label{t:main}
Let $l\ge 1$, $T$ be an $l$-dimensional vector semigroup, $(g^y\mid y\in T)$ its ergodic action on a probability space $(\Omega,\Prob)$, $S$ another vector semigroup, $\pi\maps S\to T$ a nontrivial semigroup homomorphism and $f^x=g^{\pi(x)}$ for all $x\in S$. Further let $h$ be an $f$-subadditive random function on $S$ dominated by some nonnegative random function $h_+$, such that, for all $a\in S$,

(a) $h_+(a)\in L^l(\Prob)$ and

(b) the family $(h_+(a,g^y)\mid y\in T)$ is almost independent.

\noindent Let $O$ denote the asymptotic cone of the semigroup $S$, and $q$ the gauge associated with $h$. Then there exists a measurable $W\subset \Omega$, such that \refeq{shape} holds for all $\omega\in W$ and all $(x_n)\subset S$ with asymptotic direction $x\in O$.
\end{theorem}

Although condition (b) looks bad in the context of ergodic theorems, it is satisfied in any application of Theorem~\ref{t:main} we can think of. Consider, for example, the model of first passage percolation on $\Zd^2$. For $c>0$ denote $h_c(a)=\inf_{\gamma\in\Gamma_c(0,a)}\tau_{\gamma}$, where $\Gamma_{c}(0,a)$ is the set of all $\gamma=(x_0,\dots,x_n)\in\Gamma(0,a)$, such that $\abs{x_i}\le c$ for all $i=1,\dots,n$. Since $\Gamma_c(0,a)\subset\Gamma(0,a)$, each $h_c$ dominates $h$. It is well known (see \citet{Cox-81}) that $\Mean h(a)^l<\infty$ if and only if $\Mean \min(\eps_1,\dots,\eps_4)^l<\infty$, where $\eps_1,\dots,\eps_4$ are 4 independent copies of $\eps$. The same proof applies also for $h_c(a)$ if $c$ is large enough ($c\ge\abs{a}+1$, to be more precise). Hence $\Mean h_c(a)^l<\infty\iff \Mean h(a)^l<\infty$. Moreover, $h_c(a,f^x)$ is defined by random variables $\eps_{\{y,z\}}$ with $\abs{y-x}\le c$, $\abs{z-x}\le c$. Therefore if $A,B\subset S$ and $\rho(A,B)>2c$ then $(h_c(a,f^x)\mid x\in A )$ and $(h_c(a,f^x)\mid x\in B)$ are independent families of random variables.

Of course, it is an interesting question if condition (b) in Theorem~\ref{t:main} can be dropped, and we intend to investigate it in the near future. Note, however, that the best known result for first passage percolation in $\Zd^m$ in the "non-independent" case \citep{Boivin-90} is obtained under a stronger assumption than $\Mean h(a)^m<\infty$ (although $\Mean h(a)^{m+\delta}<0$ for some $\delta>0$ is enough).

We believe that Theorem~\ref{t:main} is not only interesting as a natural generalization of the Cox-Durrett shape theorem, but can also serve as a tool for solving some open problems in the theory of first passage percolation. For example, define
\begin{equation*}
  \tilde h(x,k)=\inf_{\gamma\in\tilde\Gamma(x,k)}\tau_{\gamma}\quad\text{for $(x,k)\in\Zd^2\times\Nd$},
\end{equation*}
where $\tilde\Gamma(x,k)$ ir the set of all paths from $0$ to $x$ of length $k$. It is easily checked that
\begin{equation*}
  \tilde h(x+y,k+l,\omega)\le\tilde h(x,k,\omega)+\tilde h(y,l,f^x\omega),
\end{equation*}
therefore $\tilde h$ is a subadditive random function on $\Zd^2\times \Nd$. Moreover, Theorem~\ref{t:main} applies to it under the same moment condition $\Mean h(x)^2<\infty$, as in the Cox-Durrett shape theorem. The analogous trick proved to be useful in studying the IARCH processes, and we hope it will be helpful as well in analyzing, for example, strict convexity of the limit norm $q(x)$ in first passage percolation.

The plan of the paper is simple: in Section~\ref{s:semi} we establish some properties of vector semigroups and in Section~\ref{s:proof} prove Theorem~\ref{t:main}.

Throughout the paper $\Zd$, $\Qd$ and $\Rd$ denote, respectively, the set of all integer, rational and real numbers and $\Zd_+$, $\Qd_+$, $\Rd_+$ are their subsets consisting of nonnegative numbers. We also denote  $\Nd=\Zd_+\be\{0\}$ and call the numbers in $\Nd$ \emph{natural}. We work mainly in $\Rd^m$, and $\abs{\cdot}$ denotes some fixed norm in that space, $U(a,r)$ stands for the open ball with center $a$ and radius $r$, $[x;y]$ is the segment with endpoints $x$ and $y$. For $A\subset\Rd^m$, $\aff(A)$, $\lin(A)$, $\conv(A)$ and $\cone(A)$ denote, respectively, the affine, linear, convex and the conical hull of $A$. For convex $A$, $\rint(A)$ denotes the \emph{relative interior} of $A$, that is the interior in the space $\aff(A)$.

We constantly use the fact that the set of all linearly independent families $(a_1,\dots,a_k)$ is open in $(\Rd^m)^k$. The reason for this is that all determinants
\begin{equation}\label{e:d}
  d=\begin{vmatrix}
    a_{1j_1}&a_{2j_1}&\cdots&a_{kj_1}\\
    a_{1j_2}&a_{2j_2}&\cdots&a_{kj_2}\\
    \vdots&\vdots&\ddots&\vdots\\
    a_{1j_k}&a_{2j_k}&\cdots&a_{kj_k}\\
  \end{vmatrix}
\end{equation}
are continuous functions of $(a_1,\dots,a_k)$, and the family $(a_1,\dots,a_k)$ is linearly independent when at least one of these determinants differs from 0 (in\refeq{d}, $a_{ij}$ denotes the $j$th component of the vector $a_i$). Hence if $(a_1,\dots,a_k)$ is a basis of some linear subspace $L$ and the vectors $b_i\in L$ are close enough to $a_i$ then $(b_1,\dots,b_k)$ is also the basis of $L$. Moreover, the change of coordinates of some fixed vector $x\in L$ is arbitrary small, if the perturbation of a basis is small enough. This is because the coordinates of $x$ in the basis $(a_1,\dots,a_k)$ are equal to $d_i/d$, where
\begin{equation}\label{e:dj}
  d_i=\begin{vmatrix}
    a_{1j_1}&\cdots&a_{i-1,j_1}&x_{j_1}&a_{i+1,j_1}&\cdots&a_{kj_1}\\
    a_{1j_2}&\cdots&a_{i-1,j_2}&x_{j_2}&a_{i+1,j_2}&\cdots&a_{kj_2}\\
    \vdots&\ddots&\vdots&\vdots&\vdots&\ddots&\vdots\\
    a_{1j_k}&\cdots&a_{i-1,j_k}&x_{j_k}&a_{i+1,j_k}&\cdots&a_{kj_k}\\
  \end{vmatrix},
\end{equation}
and this determinant is also a continuous function of $(a_1,\dots,a_k)$. The same argument proves that if $x,a_1,\dots,a_k\in\Qd^m$ then all coordinates of $x$ are rational, and if $x,a_1,\dots,a_k\in\Zd^m$ then all coordinates of $x$ belong to $d^{-1}\Zd$ for some natural $d$.

\section{Vector semigroups}\label{s:semi}

\paragraph{Cones.}

Recall that a \emph{convex cone} in $\Rd^m$ is a subset $C$ with the following two properties: $x+y\in C$ and $sx\in C$ for all $x,y\in C$ and $s>0$. Two linear spaces are associated with every convex cone $C$ containing 0: the space $\lin(C)=C-C$ and the so-called \emph{lineality space} $L_0=C\cap(-C)$. The latter is the greatest linear subspace contained in $C$. If it is trivial (that is, if $L_0=\{0\}$) the cone is called \emph{pointed}.

If $A=\{a_1,\dots,a_k\}$ is a finite set then
\begin{equation*}
  \lin(A)=\Rd a_1+\cdots+\Rd a_k\quad\text{and}\quad
  \cone(A)=\Rd_+a_1+\cdots+\Rd_+a_k.
\end{equation*}
Clearly, $\cone(A)\subset\lin(A)$ and therefore $\lin(\cone(A))=\lin(A)$. Finitely generated convex cones are called \emph{polyhedral}. Each polyhedral cone is a closed set and contains 0.

For convenience, we use some concepts from convex analysis (for exact definitions and proofs see, e.g., \citet{Bruns-2009}). If $A$ is a finite set then its convex hull $P=\conv(A)$ is a \emph{polytope}. If $P$ is a polytope then the minimal set $A_0$ satisfying $P=\conv(A_0)$ is unique and is denoted by $\verts(P)$. The points in $\verts(P)$ are \emph{vertices} of $P$. If the set $\verts(P)$ is affinely independent, $P$ is called a \emph{simplex}.

The following fact is a simple implication of the theorem about stellar triangulation, but we did not find the proof in the literature so we provide our own. For the notions of a polytopal complex and a triangulation we refer again to \citet{Bruns-2009}.

\begin{lemma}\label{t:Cara}
1. Let $A$ be a finite set, $a\in A$ and $P=\conv(A)$. Then
\begin{equation*}
  P=\bigcup_{a\in B\in\calB}\conv(B),
\end{equation*}
where $\calB$ is the set of all affinely independent $B\subset A$ with $\aff(B)=\aff(A)$.

2. Let $A$ be a finite set, $0\ne a\in A$ and $C=\cone(A)$. If $C$ is pointed then
\begin{equation*}
  C=\bigcup_{a\in B\in\calB}\cone(B),
\end{equation*}
where $\calB$ is the set of all $B\subset A$ that form a basis of $\lin(A)$.
\end{lemma}

\begin{proof}
1. Let $\Pi$ be the set of all faces of $P$ and $\Pi_0=\{Q\in\Pi\mid a\not\in Q\}$. Then $\Pi$ is a polytopal complex and $\Pi_0$ its subcomplex. By Theorem~1.51 of \citet{Bruns-2009}, there exists a triangulation $\Pi_0'$ of $\Pi_0$ with $\verts(\Pi_0')=\verts(\Pi_0)$. Denote
\begin{equation*}
  \Pi'=\Pi_0'\cup\{\conv(Q',a)\mid Q'\in\Pi_0'\}.
\end{equation*}
By Lemma~1.50 of \citet{Bruns-2009}, $\Pi'$ is a triangulation of $\Pi$. Clearly, $P\not\in\Pi_0$ and therefore $Q'\subset\partial P$ for all $Q'\in\Pi_0'$. Hence
\begin{equation*}
  \rint(P)\subset\bigcup_{Q'\in\Pi_0'}\conv(Q',a).
\end{equation*}
Since the set $\Pi_0'$ is finite, the union in the right hand side is a closed set and therefore it also covers $P=\clos{\rint(P)}$ (the latter equality is valid for any closed convex $P$, see \citet[Theorem~6.3]{Rockafellar-72}). Each pyramid  $Q=\conv(Q',a)$ is a polytope, and therefore coincides with $\conv(B)$, where $B=\verts(Q)$. Clearly, $a\in B$ and $B\subset\verts(\Pi_0')\cup\{a\}=\verts(\Pi_0)\cup\{a\}\subset A$. Moreover, $B$ is affinely independent, since $Q$ is a simplex. Without loss of generality we can assume that $\aff(B)=\aff(A)$, then $B\in\calB$.

2. Without loss of generality we can assume that $0\not\in A$. By Proposition~1.21 of \citet{Bruns-2009}, there exists a linear functional $v$ with the following properties: $v(x)>0$ for all $0\ne x\in C$, the set $P=\{x\in C\mid v(x)=1\}$ is a polytope and $C=\cone(P)$. Clearly, $P=\conv(x/v(x)\mid x\in A)$. Fix any $0\ne x\in C$. Then $x/v(x)\in P$ and, by statement 1 of the lemma, $x/v(x)\in\conv(y/v(y)\mid y\in B)$ for some $B\subset A$ with the following three properties: (1) $a\in B$, (2) the family $(x/v(x)\mid x\in B)$ is affinely independent, and (3) $\aff(x/v(x)\mid x\in B)=\aff(x/v(x)\mid x\in A)$.

If $\sum_{x\in B}s_xx=0$ for some $s_x\in\Rd$ then
\begin{equation*}
  \sum_{x\in B}s_xv(x)=0\quad\text{and}\quad \sum_{x\in B}s_xv(x)\frac{x}{v(x)}=0.
\end{equation*}
By property (2), $s_xv(x)=0$ and $s_x=0$ for all $x$. Hence $B$ is linearly independent. If $z\in A$ then, by property (3), $z/v(z)=\sum_{x\in B}s_x\frac{x}{v(x)}$ with some $s_x\in\Rd$ which sum up to 1. Then $z=\sum_{x\in B}t_xx$ with $t_x=s_xv(z)/v(x)$. Therefore $A\subset\lin(B)$ and then $\lin(A)=\lin(B)$. Hence $B\in\calB$.
\end{proof}

The properties of cones that will be needed later are summarized in the following proposition.

\begin{proposition}\label{t:linealas}
Let $A$ be a finite set, $C=\cone(A)$, $L=\lin(C)$ and $L_0=C\cap(-C)$. Denote $A_0=\{x\in A\mid -x\in C\}$ and $A_1=A\be A_0$. Then:

1) $L_0=\cone(A_0)=\lin(A_0)$,

2) for all $a\in A_1$,
\begin{equation}\label{e:CLB}
  C=L_0+\bigcup_{a\in B\in\calB}\cone(B),
\end{equation}
where $\calB$ is the set of all $B\subset A_1$, such that $(x+L_0\mid x\in B)$ is a basis of quotient space $L/L_0$.
\end{proposition}

\begin{proof}
1. If $x\in A_0$ then $\pm x\in C$, hence $A_0\subset L_0$ and $\lin(A_0)\subset L_0$, because $L_0$ is a linear subspace. Conversely, let $y\in L_0$ and
\begin{equation*}
  y=\sum_{x\in A}s_xx,\quad -y=\sum_{x\in A}t_xx
\end{equation*}
with some $s_x,t_x\ge 0$. Then
\begin{equation*}
  0=\sum_{x\in A}(s_x+t_x)x.
\end{equation*}
If $s_z>0$ for some $z$ then
\begin{equation*}
  -z=\sum_{A\ni x\ne z}\frac{s_x+t_x}{s_z+t_z}x\in C.
\end{equation*}
and therefore $z\in A_0$. This means that $y=\sum_{x\in A_0}s_xx\in\cone(A_0)$. Hence
\begin{equation*}
  L_0\subset\cone(A_0)\subset\lin(A_0)\subset L_0.
\end{equation*}

2. Denote $\hat L=L/L_0$. It is a linear space with elements $\hat x=x+L_0$, $x\in L$. Clearly, $\hat x=0$ if and only if $x\in L_0$. For example, $\hat x=0$ for all $x\in A_0$, while $\hat a\ne 0$. For $B\subset L$ set $\hat B=\{\hat x\mid x\in B\}$. Then $\hat C=\cone(\hat A)=\cone(\hat A_1)$ and $\hat L=\lin(\hat A_1)=\lin(\hat C)$.

If $\pm\hat x\in \hat C$ then $\pm x\in C+L_0\subset C$, therefore $x\in L_0$ and $\hat x=0$. It means that the lineality space of $\hat C$ is trivial and the cone $\hat C$ is pointed. Then, by Lemma~\ref{t:Cara},
\begin{equation*}
  \hat C=\bigcup_{a\in B\in\calB}\cone(\hat B),
\end{equation*}
which is equivalent to \refeq{CLB}.
\end{proof}

The set $\calB$ mentioned in Proposition~\ref{t:linealas} is not so mysterious as it looks. If we start from any basis $(a_1,\dots,a_p)$ of $L_0$, extend it to a basis $(a_1,\dots,a_k)$ of $L$ and denote $B=\{a_{p+1},\dots,a_k\}$ then $\hat B$ is the basis of $\hat L$. Therefore $B\in\calB$, provided $B\subset A_1$. Such sets $B$ exist and any $B\in\calB$ can be obtained in this way.

\paragraph{Semigroups.}

Recall from the Introduction that by a \emph{vector semigroup} we call any nonempty $S\subset\Zd^m$, such that $x+y\in S$ for all $x,y\in S$. If such a semigroup contains 0, we call it a \emph{vector monoid}. Obviously, $S\cup\{0\}$ is a vector monoid for any vector semigroup $S$. We call a vector semigroup \emph{$k$-dimensional} if its linear hull is a $k$-dimensional linear space.

For $A\subset\Zd^m$, we denote by $\gp(A)$ and $\sg(A)$, respectively, the least group and the least semigroup $\supset A$. If $A=\{a_1,\dots,a_k\}$ is a finite set then
\begin{equation*}
  \gp(A)=\Zd a_1+\cdots+\Zd a_k\quad\text{and}\quad \sg(A)=\Zd_+ a_1+\cdots+\Zd_+ a_k.
\end{equation*}
Hence these sets are discrete analogues of $\lin(A)$ and $\cone(A)$. Finitely generated vector monoids are called \emph{affine monoids} in \citet{Bruns-2009}.

Two groups are associated with every vector monoid $S$: $G=S-S$ coincides with $\gp(S)$, and $G_0=S\cap(-S)$ is the greatest group contained in $S$. Clearly, $\lin(S)=\lin(G)$.

\begin{proposition}\label{t:sg}
If $A\subset\Zd^m$ is a finite set then there exists a $d\in\Nd$, such that $dx\in\sg(A)$ for all $x\in\cone(A)\cap\Zd^m$.
\end{proposition}

\begin{proof}
\emph{Step~1:} the case, where $A$ is linearly independent.

Let $A=\{a_1,\dots,a_k\}$ and $L=\lin(A)$; then $(a_1,\dots,a_k)$ is a basis of $L$. Let $x^i$ denote the coordinates of a vector $x\in L$ in that basis. If $a_{ij}$ denotes the $j$th component of $a_i$ then $x^i=d_i/d$, where $d$ and $d_i$ are given by \refeq{d}--\refeq{dj} (and $j_1<\cdots<j_k$ are chosen so that $d\ne 0$). Clearly, $\abs{d}\in\Nd$. If $x\in\Zd^m$, then all $\abs{d}x^i=\pm d_i$ are integers as well.

If $x\in\cone(A)$ then $x=s_1a_1+\cdots+s_ka_k$ for some $s_i\ge0$. Clearly, $s_i$ coincides with $x^i$, therefore $x^i\ge 0$. Hence if $x\in\cone(A)\cap\Zd^m$ then $\abs{d}x^i\in\Zd_+$ for all $i$, and $\abs{d}x\in\sg(A)$.

\emph{Step~2:} the general case.

Let $\calB$ denote the set of all linearly independent subsets of $A$. By the result of Step~1, for each $B\in\calB$ there exists a $d_B\in\Nd$, such that $d_Bx\in\sg(B)$ for all $x\in\cone(B)\cap\Zd^m$. Set $d=\prod_{B\in\calB}d_B$. If $x\in\cone(A)\cap\Zd^m$, then it follows from the Carath\'eodory theorem \citep[Theorem~1.55]{Bruns-2009} that $x\in\cone(B)$ for some $B\in\calB$. Then $d_Bx\in\sg(B)\subset\sg(A)$ and a fortiori $dx\in\sg(A)$.
\end{proof}

To move further, we need another simple lemma. Let $T$ be a partially ordered set and $T_0\subset T$. We say that $T_0$ is a \emph{minorant} of $T$ if for all $x\in T$ there exists an $x_0\in T_0$ with $x_0\le x$. We wonder if $T$ admits a finite minorant. If, for some $x\in T$, the set $\{y\in T\mid y\le x\}$ is finite then it contains a minimal (in $T$) element $x_0$, which obviously minorizes $x$. Therefore if all sets $\{y\in T\mid y\le x\}$ are finite then the set $T_{\min}$ of all minimal elements is a minorant of $T$, and it suffices to find out if it is finite.

We are only interested in the case, where $T\subset\Zd_+^I$, where $I$ is a finite set. The elements of $T$ are families $x=(x_i\mid i\in I)$ of nonnegative integers. If $y=(y_i\mid i\in I)$ is another element of $T$ then $x\le y$ means that $x_i\le y_i$ for all $i\in I$. Clearly, all sets $\{y\mid y\le x\}$, $x\in T$, are finite in this case.

\begin{lemma}\label{t:minoravimas}
For any $T\subset\Zd_+^I$, the set $T_{\min}$ is a finite minorant of $T$.
\end{lemma}

\begin{proof}
We use induction on the number of elements in $I$. If $I=\0$ then $T$ contains only one element, the empty family. Hence $T_{\min}=T$ is finite. Now consider the case, where $I$ is not empty.

Suppose the set $T_{\min}$ is infinite and fix any $a\in T$.. The set $\{x\in T\mid a\le x\}$ can contain only one element from $T_{\min}$, the $a$. Therefore one of the sets $\{x\in T_{\min}\mid x_j<a_j\}$, $j\in I$, is infinite. There are only finitely many integers between 0 and $a_j$, therefore one of the sets $T'=\{x\in T_{\min}\mid x_j=s\}$, $s\in\Zd_+$, is infinite. It is easily checked that $T'_{\min}=T'$. But the partially ordered set $T'$ is isomorphic to the set $T''=\{(x_i\mid i\ne j)\mid x\in T'\}\subset\Zd_+^{I\be\{j\}}$, because, for all $x,y\in T'$,
\begin{equation*}
  x\le y\iff (x_i\mid i\ne j)\le (y_i\mid i\ne j).
\end{equation*}
Hence $T'_{\min}$ is finite by induction. We got a contradiction.
\end{proof}

Now we can prove the structural theorem for vector semigroups. It is an analogue of Proposition~\ref{t:linealas}, and also some generalization of Gordon's lemma \citep[Lemma~2.9]{Bruns-2009}, well known in convex analysis.

\begin{proposition}\label{t:d}
Let $S$ be a vector monoid, $A\subset S$ its finite subset, $C=\cone(A)$, $S_C=S\cap C$, $G_0=S_C\cap(-S_C)$, $G=\gp(S_C)$, $L_0=C\cap(-C)$ and $L=\lin(C)$. Denote
\begin{equation*}
  A_0=\{x\in A\mid -x\in S_C\}\quad\text{and}\quad A_1=A\be A_0.
\end{equation*}
Then: 1)
\begin{gather}
  A_0=\{x\in A\mid -x\in C\},\quad L_0=\lin(A_0),\quad L=\lin(A),\label{e:A0}\\
  \sg(A_0)=\gp(A_0)\subset G_0,\quad \gp(A)\subset G;\label{e:sA0}
\end{gather}

2) if $A_1=\0$ then there exists a finite $T\subset S_C$, such that
\begin{equation*}
  S_C=G_0=G=T+\gp(A);
\end{equation*}

3) if $a\in A_1$ then
\begin{equation*}
  S_C= T+\gp(A_0)+\bigcup_{a\in B\in\calB}\sg(B),
\end{equation*}
where $T$ is some finite subset of $S_C$ and $\calB$ is the set of all $B\subset A_1$, such that the family $(x+L_0\mid x\in B)$ is a basis of the quotient space $L/L_0$.
\end{proposition}

\begin{proof}
1. If $x\in A_0$ then $x\in A$ and $-x\in S_C\subset C$. Conversely, if $x\in A$ and $-x\in C$ then, by Proposition~\ref{t:sg}, $-dx\in\sg(A)$ with some $d\in\Nd$. Then
\begin{equation*}
  -x=-dx+(d-1)x\in S_C+S_C\subset S_C
\end{equation*}
and therefore $x\in A_0$. Hence the first equality in \refeq{A0} holds true. The second one then follows from Proposition~\ref{t:linealas}, and the third is obvious.

Clearly, $\sg(A_0)\subset\gp(A_0)$, let us prove the converse relation. If $x\in A_0$ then $-x\in S\subset\Zd^m$ and $-x\in L_0=\cone(A_0)$ (by \refeq{A0} and Proposition~\ref{t:linealas}). Proposition~\ref{t:sg} then yields $-dx\in\sg(A_0)$ with some $d\in\Nd$. Thus $-x=-dx+(d-1)x\in\sg(A_0)$, i.e.\ $-A_0\subset\sg(A_0)$. It yields
\begin{equation*}
  \gp(A_0)=\sg(A_0)-\sg(A_0)=\sg(A_0)+\sg(-A_0)\subset\sg(A_0)+\sg(A_0)\subset\sg(A_0).
\end{equation*}

We proved that $\sg(A_0)=\gp(A_0)$. Hence this set is a group which is contained in $S_C$. Therefore $\sg(A_0)\subset G_0$. The last relation in \refeq{sA0} is obvious.

2. We follow the lines of the proof of Gordon's lemma. Let $A=A_0=\{a_1,\dots,a_k\}$. If $x\in S_C$ then $x=s_1a_1+\cdots+s_ka_k$ with some $s_1,\dots,s_k\in\Rd$. Thus
\begin{equation}\label{e:xaz}
  x=\floor{s_1}a_1+\cdots+\floor{s_k}a_k+z
\end{equation}
with a $z$ from some bounded, and hence finite, set $R\subset G$. So $S_C\subset \gp(A)+R$.

Let $T$ be a finite set, which intersects each nonempty $S_C\cap(\gp(A)+z)$, $z\in R$. If $x\in S_C$ then $x=y+z$ with some $y\in\gp(A)$ and $z\in R$. The intersection $S_C\cap(\gp(A)+z)$ is nonempty, therefore it contains some $x_0\in T$. Let $x_0=y_0+z$ with $y_0\in\gp(A)$, then $x-x_0=y-y_0\in\gp(A)$ and $x\in T+\gp(A)$. Hence $S_C\subset T+\gp(A)$. The converse relation also holds true, because \refeq{sA0} implies $\gp(A)=\gp(A_0)\subset G_0\subset S_C$. Hence $S_C=T+\gp(A)$.

It remains to prove that $S_C$ is a group: then it will coincide both with $G_0$ and $G$. Fix an arbitrary $z\in T$, then $z\in S_C$ and $nz\in S_C$ for all $n\in\Nd$. Let $nz=y_n+z_n$ with $y_n\in\gp(A)$ and $z_n\in T$. Since $T$ is finite, there exist $n_1<n_2$, such that $z_{n_1}=z_{n_2}$. Then for $n=n_2-n_1\ge 1$ we get
\begin{equation*}
  nz=n_2z-n_1z=y_{n_2}+z_{n_2}-y_{n_1}-z_{n_1}=y_{n_2}-y_{n_1}\in\gp(A).
\end{equation*}
Hence $-z=(n-1)z-nz\in(n-1)z+\gp(A)\subset S_C$.

We have proved that $-T\subset S_C$. Then $-S_C=-T-\gp(A)\subset S_C$ and therefore $S_c$ is a group.

3. For short, denote $\calB_a=\{B\in\calB\mid a\in B\}$. If $x\in S_C$ then $x\in C$ and it follows from Proposition~\ref{t:linealas} that $x\in L_0+\cone(B)$ with some $B\in\calB_a$. Let $(a_1,\dots,a_p)\subset A_0$ be a basis of $L_0$ and $B=\{a_{p+1},\dots,a_k\}$. Then $x=s_1a_1+\cdots+s_ka_k$ with some $s_i\in\Rd$; moreover, $s_i\ge 0$ for $i>p$. Again, \refeq{xaz} holds with a $z$ from some bounded, and hence finite, set $R(B)\subset G$. So $x\in\gp(A_0)+\sg(B)+R(B)$.

For $B\in\calB_a$ and $z\in R(B)$ set $S(B,z)=S_C\cap(\sg(B)+z)$. If $B=\{b_1,\dots,b_q\}$ then each $x\in S(B,z)$ has a unique representation of the form $x=z+x^1b_1+\cdots+x^qb_q$, where $x^i\in\Zd_+$. For $x,y\in S(B,z)$ let us write $x\le y$ if $x^i\le y^i$ for all $i$. Then $\le$ is a partial order on $S(B,z)$, and that partially ordered set is isomorphic to $\Zd_+^q$. It follows from Lemma~\ref{t:minoravimas} that there exists a finite minorant $T(B,z)$ of $S(B,z)$. Clearly, $T(B,z)\subset S_C$ and $S(B,z)\subset T(B,z)+\sg(B)$.

Denote $T=\bigcup_{B\in\calB_a, z\in R(B)}T(B,z)$; then $T$ is a finite subset of $S_C$. If $x\in S_C$ then, for some $y\in\gp(A_0)$, $B\in\calB_a$ and $z\in R(B)$,
\begin{equation*}
  x-y\in S(B,z)\subset T(B,z)+\sg(B)\subset T+\sg(B).
\end{equation*}
Hence $S_C\subset T+\gp(A_0)+\bigcup_{B\in\calB_a}\sg(B)$. The converse relation is obvious.
\end{proof}

The following example shows that relations $\subset$ in \refeq{sA0} cannot be replaced by equalities.

\begin{example}
Let $S=\Zd^2$ and $A$ consists of 4 vectors $(\pm1,\pm1)$. Then $C=\Rd^2$, $S_C=S=\Zd^2$ and therefore $G=G_0=\Zd^2$. On the other hand, $A_0=A$ and
\begin{equation*}
  \sg(A_0)=\gp(A_0)=\{(x_1,x_2)\in\Zd^2\mid x_1\equiv x_2\pmod*{2}\}.
\end{equation*}
\end{example}

\paragraph{The asymptotic cone of a semigroup.}

Let $S$ be a vector semigroup and $L=\lin(S)$. Denote
\begin{equation*}
  S^*=\bigcup_{n\ge 1}S/n.
\end{equation*}
Clearly, $S^*$ is a semigroup too, although not a vector semigroup, because $S^*\not\subset\Zd^m$. However, $S^*\subset\Qd^m$ and $sx\in S^*$ for all $x\in S^*$ and positive $s\in\Qd$. Obviously, $S\subset S^*\subset L$ and therefore $L=\lin(S^*)$.

The set $C=\cone(S)$ is a convex cone, therefore $\aff(C)=\lin(C)=\lin(S)=L$. It is well-known that then $O=\rint(C)$ is also a convex cone, moreover $\lin(O)=L$ \citep[Theorem~6.2]{Rockafellar-72}. We call $O$ the \emph{asymptotic cone} of $S$. If $x\in O$, then $x=s_1x_1+\cdots+s_kx_k$ with some $x_1,\dots,x_k\in S$ and $s_1,\dots,s_k\in\Rd_+$. If $t_i$ are rational numbers close enough to $s_i$ then the vector $y=t_1x_1+\cdots+t_kx_k\in S^*$ is arbitrary close to $x$. Hence $O\subset\clos{S^*}$.

Our next goal is to show that every point $x\in O$ lies in the relative interior of some full-dimensional simplex with vertices in $S^*$, and derive some corollaries from that. To this end, we introduce some more notation. If $P$ is a simplex and $\aff(P)=\lin(P)=L$ we call it an \emph{$L$-simplex}. An $L$-simplex is called an \emph{$S$-simplex} if all its vertices belong to $S^*$.

Let $P$ be an $L$-simplex with vertices $a_0,\dots,a_k$, then $a_i\in L$ for all $i$ and $(a_1-a_0,\dots,a_k-a_0)$ is a basis of $L$. Since $P-a_0$ is the image of the set $\{(s_1,\dots,s_k)\in\Rd_+^k\mid s_1+\cdots+s_k\le 1\}$ by the homeomorphism $(s_1,\dots,s_k)\mapsto \sum_{i=1}^ks_i(a_i-a_0)$, we get
\begin{equation*}
  \rint(P)=\{\tsum_{i=0}^ks_ia_i\mid s_0,\dots,s_k>0,\ \tsum_{i=0}^ks_i=1\}
\end{equation*}
If $x\in\rint(P)$ then $x-a_0=\tsum_{i=1}^ks_i(a_i-a_0)$ with some $s_i>0$, such that $\sum_{i=1}^ks_i<1$.

Now let $b_i\in L$ and $\abs{b_i-a_i}<\delta$ for $i=0,\dots,k$. If $\delta$ is small enough then $(b_1-b_0,\dots,b_k-b_0)$ is another basis of $L$. Moreover, the coordinates of $x-b_0$ in basis $(b_1-b_0,\dots,b_k-b_0)$ are close to $s_1,\dots,s_k$ and therefore they are positive and their sum is less than 1. Hence if $\delta$ is small enough then $Q=\conv(b_0,\dots,b_k)$ is another $L$-simplex and $x\in\rint(Q)$.

Consequently, the following statement holds true: if $P$ is an $L$-simplex with vertices in $O$ and $x\in\rint(P)$ then there exists an $S$-simplex $Q$ with $x\in\rint(Q)$.

\begin{proposition}\label{t:simpleksas}
Let $S$ be a vector semigroup and $O$ its asymptotic cone. Then:

1) each $x\in O$ lies in the relative interior of some $S$-simplex,

2) if $0\in O$ then $S$ is a group,

3) $O\cap\Qd^m\subset S^*$.
\end{proposition}

\begin{proof}
1. Let $x\in O$, $L=\lin(S)$, $(e_1,\dots,e_k)$ be a basis of $L$ and
\begin{equation*}
  a_0=x-\epsilon(e_1+\cdots+e_k),\quad a_i=x+\epsilon e_i\text{ for $i=1,\dots,k$,}
\end{equation*}
where $\epsilon$ is so small that $a_i\in O$ for all $i$. If
\begin{equation*}
  \sum_{i=0}^ks_ia_i=0\quad\text{and}\quad \sum_{i=0}^ks_i=0
\end{equation*}
then
\begin{equation*}
  0=\sum_{i=0}^ks_ix+\epsilon\sum_{i=1}^k(s_i-s_0)e_i=\epsilon\sum_{i=1}^k(s_i-s_0)e_i,
\end{equation*}
which yields $s_i=s_0$ for $i=1,\dots,k$. Then $0=\sum_{i=0}^ks_i=(k+1)s_0$ and therefore $s_i=0$ for all $i$. Hence the family $(a_0,\dots,a_k)$ is affinely independent and $P=\conv(a_0,\dots,a_k)$ is an $L$-simplex. Since
\begin{equation*}
  x=\frac{a_0+\cdots+a_k}{k+1},
\end{equation*}
$x\in\rint(P)$. It remains to apply the statement just before the proposition.

2. Let $0\in O$ and $P$ be an $S$-simplex with vertices $a_0,\dots,a_k$, such that $0\in\rint(P)$. Then $0=s_0a_0+\cdots+s_ka_k$ with some positive $s_i$ whose sum equals 1. The equality remains valid if we multiply it by any natural number, therefore without lost of generality we can assume that $a_i\in S$ for all $i$. Since $(a_1-a_0,\dots,a_k-a_0)$ is a basis of $L$ and
\begin{equation*}
  -a_0=s_1(a_1-a_0)+\cdots+s_k(a_k-a_0),
\end{equation*}
all $s_i$ are rational. Therefore multiplying once again the initial equality by some natural number we get, for some $n_0,\dots,n_k\in\Nd$,
\begin{equation*}
  n_0a_0+\cdots+n_ka_k=0.
\end{equation*}

Denote $C=\cone(a_0,\dots,a_k)$. If $x\in C$ then $x=\sum_{i=0}^ks_ia_i$ with some $s_i\ge 0$. Then, for some natural $t$ large enough,
\begin{equation*}
  -x=t0-x=\sum_{i=0}^k(tn_i-s_i)a_i\in C.
\end{equation*}
Hence $-C\subset C$, i.e.\ $C$ is a linear space --- coincides with $L=\lin(C)$. By Proposition~\ref{t:d}, $S$ is a group.

3. Let $x\in O\cap\Qd^m$ and $P$ be an $S$-simplex with vertices $a_0,\dots,a_k$, such that $x\in\rint(P)$. Suppose $x=s_0a_0+\cdots+s_ka_k$ with positive $s_i$ whose sum equals 1. It follows from
\begin{equation*}
  \Qd^m\ni x-a_0=\sum_{i=1}^ks_i(a_i-a_0)
\end{equation*}
that $s_1,\dots,s_k\in\Qd$. Clearly, then also $s_0=1-\sum_{i=1}^ks_i$ is rational. Hence $x\in S^*$.
\end{proof}

\paragraph{Homomorphisms.}

Let $S$ and $T$ be two vector semigroups and $\pi\maps S\to T$ a semigroup homomorphism, that is $\pi(x+y)=\pi(x)+\pi(y)$ for all $x,y\in S$. Clearly, then $\pi(tx)=t\pi(x)$ for all $x\in S$ and $t\in\Nd$. If $0\in S$ then
\begin{equation*}
  \pi(0)=\pi(0+0)=\pi(0)+\pi(0),
\end{equation*}
which implies $\pi(0)=0$. Hence in this case $T$ is a vector monoid as well. If $0\not\in S$, we can extend $\pi$ to a homomorphism from $S\cup\{0\}$ to $T\cup\{0\}$ by setting $\pi(0)=0$.

\begin{proposition}\label{t:shomo}
Let $S$ and $T$ be vector semigroups and $\pi\maps S\to T$ a semigroup homomorphism. Then $\pi$ is extended to a linear operator from $L=\Lin(S)$ to $\lin(T)$.
\end{proposition}

\begin{proof}
Without loss of generality we can assume that both $S$ and $T$ are monoids and $\pi(0)=0$. Let $(a_1,\dots,a_k)\subset S$ be a basis of $L$ and $u$ the unique linear operator from $L$ to $\lin(T)$ which maps $a_i$ to $\pi(a_i)$ for $i=1,\dots,k$. Let $x\in S$, then $x=\sum_{i=1}^ks_ia_i$ with some $s_i\in d^{-1}\Zd$, where $d$ is some natural number. Set $t_i=ds_i$. Then
\begin{gather*}
  dx+\sum_{i=1}^kt_i^-a_i=\sum_{i=1}^kt_i^+a_i,\\
  d\pi(x)+\sum_{i=1}^kt_i^-\pi(a_i)=\sum_{i=1}^kt_i^+\pi(a_i),\\
  d\pi(x)+\sum_{i=1}^kt_i^-u(a_i)=\sum_{i=1}^kt_i^+u(a_i)
\end{gather*}
and therefore
\begin{equation*}
  \pi(x)
  =\frac1d\Bigl(\sum_{i=1}^kt_i^+u(a_i)-\sum_{i=1}^kt_i^-u(a_i)\Bigr)
  =\frac1d\sum_{i=1}^kt_iu(a_i)
  =\sum_{i=1}^ks_iu(a_i)
  =u(x).
\end{equation*}
\end{proof}

In the sequel we will denote the extended operator by the same letter $\pi$, and $\norm{\pi}$ will stand for its norm. Hence $\abs{\pi(x)}\le\norm{\pi}\,\abs{x}$ for all $x\in S$.

\paragraph{$S$-cones.}

Let $S$ be a vector semigroup and $L=\lin(S)$. By \emph{$S$-cone} we call any set of the form $C=\cone(a_1,\dots,a_k)$, where $(a_1,\dots,a_k)\subset S$ is a basis of $L$. Such a cone is the image of $\Rd_+^k$ by the homeomorphism $(s_1,\dots,s_k)\mapsto \sum_{i=1}^ks_ia_i$, therefore $\lin(C)=L$ and
\begin{equation*}
  \rint(C)=\{\tsum_{i=1}^ks_ia_i\mid s_1,\dots,s_k>0\}.
\end{equation*}
Moreover, equality $-a_j=\sum_{i=1}^ks_ia_i$ implies $s_j=-1$, therefore $-a_j\not\in C$ for all $j$, and, by Proposition\ref{t:linealas}, $C$ is pointed.

If $(a_1,\dots,a_k)\subset S^*$ is a basis of $L$ then $n_ia_i\in S$ with some natural $n_i$. Then $\cone(a_1,\dots,a_k)=\cone(n_1a_1,\dots,n_ka_k)$ is an $S$-cone.

\begin{proposition}\label{t:vidiniskugis}
Let $S$ be a $k$-dimensional vector semigroup and $O$ its asymptotic cone. For each $x\in O\be\{0\}$ there exists an $S$-cone $C=\cone(a_1,\dots,a_k)$ with $x\in\rint(C)$. If $\pi$ is a nontrivial homomorphism from $S$ to another vector semigroup $T$ then $a_i$ can be chosen so that $\pi(a_i)\ne 0$ for all $i$.
\end{proposition}

\begin{proof}
Let $L=\lin(S)$ and $x\in O\be\{0\}$. It follows from the Carath\'eodory theorem \citep[Theorem~1.55]{Bruns-2009} that $x=s_1a_1+\cdots+s_ka_k$ with some linearly independent $(a_1,\dots,a_k)\subset O$ and some $s_1,\dots,s_k\ge 0$. Since $x\ne 0$, at least one of coordinates $s_i$ differs from 0. Let, for instance, $s_1>0$. Denote $a=a_1-\epsilon(a_2+\cdots+a_k)$, where $\epsilon$ is so small that $a\in O$. Clearly, $(a,a_2,\dots,a_k)$ is another basis of $L$. If $t_1,\dots,t_k$ are the coordinates of $x$ in that basis, then
\begin{equation*}
  \sum_{i=1}^ks_ia_i=t_1a+\sum_{i=2}^kt_ia_i=t_1a_1+\sum_{i=2}^k(t_i-\epsilon t_1)a_i,
\end{equation*}
which implies $t_1=s_1>0$ and $t_i=s_i+\epsilon s_1>0$ for $i=2,\dots,k$. Now let $b_1,\dots,b_k$ be the vectors from $S^*$, such that $\abs{b_1-a}<\delta$ and $\abs{b_i-a_i}<\delta$ for $i=2,\dots,k$. If $\delta$ is small enough then $(b_1,\dots,b_k)$ is yet another basis of $L$ and all coordinates of $x$ in that basis are positive. Hence $C=\cone(b_1,\dots,b_k)$ is an $S$-cone and $x\in\rint(C)$.

Let $\pi$ be a nontrivial homomorphism from $S$ to $T$ and $\pi(a_0)\ne 0$ for some $a_0\in S$. Define $b'_i=b_i$ if $\pi(b_i)\ne 0$, and $b'_i=nb_i+a_0$ if $\pi(a_i)=0$, where $n$ is a big natural number. Then $\pi(b'_i)\ne 0$ for all $i$. Moreover,
\begin{equation*}
  \frac{b'_i}{\abs{b'_i}}
  =\frac{nb_i+O(1)}{n\abs{b_i}+O(1)}
  =\frac{b_i}{\abs{b_i}}+O(1/n),
\end{equation*}
as $n\to\infty$; therefore, for $n$ large enough, the family $(b'_1,\dots,b'_k)$ is linearly independent and $x\in\rint(\cone(b'_1,\dots,b'_k))$.
\end{proof}

\paragraph{Gauges.}

If $C$ is a convex cone and $q\maps C\to[-\infty;\infty)$, we call the function $q$ \emph{a gauge} if, for all $x,y\in C$ and $s>0$,
\begin{equation}\label{e:gauge}
  q(x+y)\le q(x)+q(y)\quad\text{and}\quad q(sx)=sq(x).
\end{equation}
If $C$ is merely a semigroup and condition~\refeq{gauge} holds for all natural $s$, we call $q$ a \emph{$\Zd$-gauge} on $C$. If $C$ is a semigroup and $sx\in C$ for all $x\in C$ and rational $s>0$, we call $C$ \emph{a $\Qd$-cone}. In that case $q$ is called \emph{a $\Qd$-gauge}  if condition~\refeq{gauge} holds for all rational $s>0$.

Note that if $S$ is a vector semigroup then $S^*$ is a $\Qd$-cone.

\begin{proposition}\label{t:kalibras}
Let $q$ be a $\Zd$-gauge on a vector semigroup $S$ and $O$ the asymptotic cone of $S$. Then

1) the function $q$ is extended in a unique way to a $\Qd$-gauge on $S^*$;

2) the restriction of $q$ on $O\cap S$ is extended in a unique way to a gauge on $O$.
\end{proposition}

\begin{proof}
\emph{Step~1:} we prove statement 1.

For each $x\in S^*$ there exists a natural $k$, such that $kx\in S$. If $lx\in S$ with another natural $l$ then
\begin{equation*}
  lq(kx)=q(klx)=kq(lx),
\end{equation*}
which implies
\begin{equation*}
  \frac{q(kx)}{k}=\frac{q(lx)}{l}.
\end{equation*}
Therefore the following definition is correct:
\begin{equation*}
  q^*(x)=\frac{q(kx)}{k}\quad\text{if $kx\in S$.}
\end{equation*}

Let $x\in S^*$, $0<t\in\Qd$ and $kx\in S$, $lt\in\Nd$ with some natural $k$ and $l$. Then $kltx\in S$ and therefore
\begin{equation*}
  q^*(tx)=\frac{q(kltx)}{kl}=\frac{ltq(kx)}{kl}=tq^*(x).
\end{equation*}
If $x,y\in S^*$ then $kx\in S$ and $ly\in S$ with some natural $k$ and $l$. Then also $kl(x+y)=l(kx)+k(ly)\in S$ and therefore
\begin{equation*}
  q^*(x+y)=\frac{q(klx+lky)}{kl}\le\frac{lq(kx)+kq(ly)}{kl}=q^*(x)+q^*(y).
\end{equation*}
Hence $q^*$ is a $\Qd$-gauge on $S^*$.

Obviously, $q^*(x)=q(x)$ for $x\in S$ and $q^*$ is the unique $\Qd$-gauge on $S^*$ which extends $q$.

The remaining proof is similar to that of continuity of a convex function. We denote $L=\lin(S)$, $k=\dim{L}$ and speaking about a neighborhood of some $x\in L$ we mean a neighborhood in $L$.

\emph{Step~2:} we prove that for each $x_0\in O$ there exists its convex neighborhood $U\subset O$, such that $q^*$ is bounded from above on $U\cap S^*$.

Let $U=\rint(P)$, where $P$ is an $S$-simplex with vertices $a_0,\dots,a_k$, such that $x_0\in\rint(P)$. If $x\in U\cap S^*$ then $x=s_0a_0+\cdots+s_ka_k$ with some positive $s_i$ whose sum equals 1. Since the family $(a_1-a_0,\dots,a_k-a_0)$ is linearly independent, equality
\begin{equation*}
  x-a_0=s_1(a_1-a_0)+\cdots+s_k(a_k-a_0)
\end{equation*}
implies that $s_1,\dots,s_k$ are rational numbers. Clearly, then $s_0$ is rational, too. Therefore
\begin{equation*}
  q^*(x)\le\sum_{i=0}^ks_iq^*(a_i)\le c,
\end{equation*}
where $c=\max_iq^*(a_i)$.

\emph{Step~3:} the case, where $q^*(x_n)\to-\infty$ for some sequence $S^*\ni x_n\to x\in O$.

Let $U\subset O$ be a convex neighborhood of $x$, such that $q^*$ is bounded from above on $U\cap S^*$. Fix any $y\in U\cap S^*$, find a rational $\epsilon$, such that $y-\epsilon(x-y)\in U$, and set $z_n=y-\epsilon(x_n-y)$. If $n$ is large enough then $z_n\in U\cap\Qd^m$, therefore $z_n\in S^*$ (by Proposition~\ref{t:simpleksas}) and the sequence $q^*(z_n)$ is bounded from above. Then $y=(z_n+\epsilon x_n)/(1+\epsilon)$ implies
\begin{equation*}
  q^*(y)\le\frac{q^*(z_n)+\epsilon q^*(x_n)}{1+\epsilon}\to-\infty,
\end{equation*}
i.e.\ $q^*(y)=-\infty$.

Now let $y'$ be an arbitrary point in $O\cap S^*$. The segment $I=[y;y']$ is a compact set covered by sets $U(x')$, $x'\in I$, where $U(x')\subset O$ is a convex neighborhood of $x'$, such that $q^*$ is bounded from above on $U(x')\cap S^*$. Hence there exists a finite covering of $I$ by the sets $U(x')$. Let $U_1\cups U_k$ be such a covering, and $y\in U_1$, $y'\in U_k$, $U_i\cap U_{i+1}\ne\0$ for $i=1,\dots,k-1$.

Choose a $y_i\in S^*$ in each $U_i\cap U_{i+1}$ and denote $y_0=y$, $y_k=y'$. Then $[y_{i-1};y_i]\subset U_i$ for all $i=1,\dots,k$. We know already that $q^*(y_0)=-\infty$, let us prove that $q^*(y_i)=-\infty$ for all $i$. Let $i\ge 1$ and suppose $q^*(y_{i-1})=-\infty$. Find a rational $\epsilon$ small enough, so that $z=y_i-\epsilon(y_{i-1}-y_i)$ lies in $U_{i}$. Since $z\in\Qd^m$, it belongs also to $S^*$. Then $y_i=(z+\epsilon y_{i-1})/(1+\epsilon)$ implies
\begin{equation*}
  q^*(y_i)\le\frac{q^*(z)+\epsilon q^*(y_{i-1})}{1+\epsilon}=-\infty,
\end{equation*}
i.e.\ $q^*(y_i)=-\infty$.

Hence $q^*(y')=-\infty$. We thus proved that in the considered case $q^*(x')=-\infty$ for any $x'\in O\cap S^*$. Define $\bar q(x')=-\infty$ for all $x'\in O$. Then $\bar q$ is a gauge on $O$, which extends $q^*$. The extension is unique, because every convex function on $O$, which equals $-\infty$ at some point, is identically equal to $-\infty$.

\emph{Step~4:} the remaining case.

Suppose that there is no sequence $(x_n)\subset S^*$ with $x_n\to x\in O$ and $q^*(x_n)\to-\infty$. In that case $q^*$ is bounded from below in each compact subset of $O$ and therefore, by the result of Step~2, is bounded in some neighborhood of each point of $O$. We prove that $q^*$ is locally Lipschitz.

Let $x_0\in O$. Find $\epsilon$ and $c$, such that $U(x_0,2\epsilon)\cap L\subset O$ and $\abs{q^*(x)}\le c$ for all $x\in U(x_0,2\epsilon)\cap S^*$. Take any two different $x,y\in U(x_0,\epsilon)\cap S^*$ and denote $k=\floor{\epsilon/\abs{x-y}}$ and $z=(k+1)x-ky$. Since
\begin{equation*}
  \abs{z-x_0}\le\abs{x-x_0}+k\abs{x-y}<\epsilon+\epsilon=2\epsilon,
\end{equation*}
we get $z\in U(x_0,2\epsilon)\subset O$. Moreover, $z\in\Qd^m$ and therefore $z\in S^*$. Then $x=(ky+z)/(k+1)$ implies
\begin{equation*}
  q^*(x)-q^*(y)
  \le\frac{kq^*(y)+q^*(z)}{k+1}-q^*(y)
  =\frac{q^*(z)-q^*(y)}{k+1}
  \le\frac{2c}{k+1}
  \le\frac{2c}{\epsilon}\abs{x-y}.
\end{equation*}
Because of symmetry, the analogous inequality with $q^*(y)-q^*(x)$ on the left hand side also holds. Hence
\begin{equation*}
  \abs{q^*(x)-q^*(y)}\le\frac{2c}{\epsilon}\abs{x-y}.
\end{equation*}

It is well known that any uniformly continuous (and hence any Lipschitz) real function defined on a dense subset of a metric space $E$ is extended in a unique way to a continuous function defined on the whole $E$. Therefore there exists an open covering $(U_i)$ of $O$ and, for each $i$, a continuous function $\bar q_i\maps U_i\to\Rd$, which agrees with $q^*$ on $U_i\cap S^*$. By continuity, any two functions $\bar q_i$ and $\bar q_j$ agree on the intersection $U_i\cap U_j$. Therefore there exists a continuous function $\bar q\maps O\to\Rd$, which extends the restriction of $q^*$ on $O\cap S^*$.

For all $x',y'\in S^*$ and positive $t'\in\Qd$,
\begin{equation*}
  q^*(x'+y')\le q^*(x')+q^*(y')\quad\text{and}\quad
  q^*(t'x')=t'q^*(x').
\end{equation*}
Taking the limits, as $x'\to x\in O$, $y'\to y\in O$ and $t'\to t>0$, yields that $\bar q$ is a gauge on $O$. The extension $\bar q$ is unique, because every convex function on $O$ is continuous.
\end{proof}

\section{Subadditive ergodic theorem}\label{s:proof}

\paragraph{Almost independence.}

Recall from the Introduction that a family of random variables $(Z_x\mid x\in S)$ is called \emph{almost independent}, if there exists a $c<\infty$, such that, for all nonempty $A,B\subset S$ with $\rho(A,B)>c$, the families $(Z_x\mid x\in A)$ and $(Z_x\mid x\in B)$ are independent. Here $S\subset\Zd^m$ and $\rho(A,B)=\inf_{x\in A,y\in B}\abs{x-y}$ is the usual distance between sets $A$ and $B$. It is easily checked that the notion of almost independence does not depend on the norm $\abs{\cdot}$ on $\Rd^m$.

A sequence $(Z_i\mid i\ge 1)$ is almost independent when it is $l$-dependent for some natural $l$, that is, if for all $n$ the families $(X_1,\dots,X_n)$ and $(X_i\mid i\ge n+l)$ are independent.

If $(Z_i)$ is a sequence of iid zero-mean random variables and $S_n=Z_1+\cdots+Z_n$, then almost surely
\begin{equation*}
  S_n-n\epsilon=n\Bigl(\frac{S_n}{n}-\epsilon\Bigr)\to -\infty,
\end{equation*}
and therefore the random variable
\begin{equation}\label{e:M}
  M=\sup_{n\ge 0}(S_n-n\epsilon)
\end{equation}
is almost surely finite. It is also known \citep[Theorem~4.6.1 (iii)]{Borovkov-2008} that, for $k\ge2$, $\Mean\abs{Z}^k<\infty$ implies $\Mean M^{k-1}<\infty$. The following proposition extends this statement to the case, where $(Z_i)$ is almost independent.

\begin{proposition}\label{t:M}
Let $Z$ be a random variable with $\Mean\abs{Z}^k<\infty$ and $\Mean Z=0$, where $k\ge 2$ is a natural number. If $(Z_i)$ is an almost independent sequence of copies of $Z$, $S_n=Z_1+\cdots+Z_n$ and $M$ is defined by \refeq{M} then $\Mean M^{k-1}<\infty$.
\end{proposition}

\begin{proof}
Suppose that the sequence $(Z_i)$ is $l$-dependent. For each $r=1,\dots,l$ define
\begin{equation*}
  M_r=\sup_{j\ge 0}(Z_r+Z_{r+l}+\cdots+Z_{r+(j-1)l}-j\epsilon)=\sup_{j\ge 0}\sum_{i=0}^{j-1}(Z_{r+il}-\epsilon).
\end{equation*}
If $n=sl+r'$ with $s\in\Zd_+$ and $1\le r\le r'$ then
\begin{equation*}
  \sum_{i=1}^nZ_i-n\epsilon
  =\sum_{r=1}^{r'}\sum_{i=0}^{s}(Z_{r+il}-\epsilon)+\sum_{r=r'+1}^l\sum_{i=0}^{s-1}(Z_{r+il}-\epsilon)
  \le\sum_{r=1}^lM_{r},
\end{equation*}
and therefore $M\le\sum_{r=1}^lM_{r}$. Since each sequence $(Z_{r+il}\mid i\ge 0)$ consists of iid random variables, $M_r\in L^{k-1}(\Prob)$ for all $r$. Hence $M\in L^{k-1}(\Prob)$.
\end{proof}

Let $S$ be a vector semigroup. From now on we look at it as the locally compact metric space, $S\cup\{\infty\}$ being its one-point compactification. Hence if $(x_n)\subset S$ then $x_n\to\infty$ means $\abs{x_n}\to\infty$. If $h$ is some real function on $S$ then $h(x)\tends{x\to\infty}c$ means that $h(x_n)\to c$ for any sequence $(x_n)\subset S$ with $\abs{x_n}\to\infty$.

\begin{proposition}\label{t:max}
Let $k\ge 1$, $S$ be a $k$-dimensional vector semigroup, $Z$ a nonnegative random variable with $\Mean Z^k<\infty$ and $(Z_x\mid x\in S)$ an almost independent family of its copies. Then almost surely
\begin{equation*}
  \frac{Z_{x}}{\abs{x}}\to 0,\quad\text{as $S\ni x\to\infty$}.
\end{equation*}
\end{proposition}

\begin{proof}
Let $(a_1,\dots,a_k)\subset S$ be a basis of $L=\lin(S)$ and $x^i$ denote the coordinates of a vector $x\in L$ in this basis. Let $d$ be a natural number, such that $dx^i\in\Zd$ for all $x\in S$ and all $i$. For $x\in L$ define $\norm{x}=\max_i\abs{dx^i}$. It is another norm on $L$, which takes integer values for $x\in S$. All norms in a finite-dimensional space are equivalent, therefore it suffices to prove that $Z_x/\norm{x}\to 0$, as $S\ni x\to\infty$.

Set $S_n=\{x\in S\mid\norm{x}=n\}$. Then
\begin{equation}\label{e:Sn}
  \abs{S_n}
  \le\abs{\{(s_1,\dots,s_k)\in\Zd^k\mid\max_i\abs{s_i}=n\}}
  =O(n^{k-1}),
\end{equation}
as $n\to\infty$. Next, fix $\delta$ and denote $p_n=\Prob(Z>n\delta)$. Then
\begin{multline}
  \sum_{n\ge 1}n^{k-1}p_n
  =\sum_{n\ge 1}n^{k-1}\sum_{i\ge n}\Prob(i<Z/\delta\le i+1)\\
  =\sum_{i\ge 1}\Prob(i<Z/\delta\le i+1)\sum_{n=1}^in^{k-1}
  \le\sum_{i\ge 1}i^k\Mean\1_{i<Z/\delta\le i+1}
  \le\Mean(Z/\delta)^k
  <\infty.\label{e:npn}
\end{multline}

For $A,B\subset S$ denote $\bar\rho(A,B)=\inf_{x\in A,y\in B}\norm{x-y}$, and let $l$ be a natural number, such that $\bar\rho(A,B)\ge l$ implies independence of families $(Z_x\mid x\in A)$ and $(Z_x\mid x\in B)$. For $y\in S$ define
\begin{equation*}
  S_n(y)=\{x\in S_n\mid dx\equiv dy\pmod*{l}\},
\end{equation*}
where  $dx\equiv dy\pmod{l}$ mean that $dx^i\equiv dy^i\pmod{l}$ for all $i=1,\dots,k$. It is easily seen that there exists a finite $B\subset S$, such that $S_n=\bigcup_{y\in B}S_n(y)$ for all $n$.

If $x,x'\in S$, $dx\equiv dx'\pmod {l}$ and $x\ne x'$, then $\norm{x-x'}\ge l$. Therefore each family $(Z_x\mid x\in S_n(y))$ consists of independent random variables. Then, for $n$ large enough,
\begin{multline*}
  \Prob\Bigl(\max_{x\in S_n}Z_x>\delta n\Bigr)
  \le\sum_{y\in B}\Prob\Bigl(\max_{x\in S_n(y)}Z_x>\delta n\Bigr)
  \le\sum_{y\in B}(1-(1-p_n)^{\abs{S_n(y)}})\\
  \le\abs{B}(1-(1-p_n)^{\abs{S_n}})
  \le\abs{B}(1-\e^{-\abs{S_n}p_n/2})
  \le\abs{B}\,\abs{S_n}p_n/2,
\end{multline*}
and \refeq{Sn}--\refeq{npn} imply
\begin{equation*}
  \sum_{n\ge 1}\Prob\Bigl(\max_{x\in S_n}Z_x>\delta n\Bigr)<\infty.
\end{equation*}

By the Borel-Cantelli lemma, almost surely only finite number of events $\{\max_{x\in S_n}Z_x>\delta n\}$ occur. Therefore almost surely
\begin{equation*}
  \varlimsup_{n\to\infty}\frac{1}{n}\max_{x\in S_n}Z_x\le\delta.
\end{equation*}
The obtained inequality is valid for any $\delta$, therefore, for almost all $\omega$,
\begin{equation*}
  \frac{1}{n}\max_{\norm{x}=n}Z_x(\omega)\to 0.
\end{equation*}
Fix $\omega$, for which the latter relation holds. Then $\norm{x_n}\to\infty$ implies
\begin{equation*}
  \frac{Z_{x_n}(\omega)}{\norm{x_n}}
  \le\frac{1}{\norm{x_n}}\max_{\norm{x}=\norm{x_n}}Z_x(\omega)\to 0.
\end{equation*}
\end{proof}

\paragraph{The maximal ergodic theorem.}

Recall from Introduction that if $\Omega$ is a measurable space, then any measurable $f\maps\Omega\to\Omega$ is called a \emph{transformation} of $\Omega$. If $f$ is a transformation, we write $f\omega$ instead of $f(\omega)$, and if $g$ is another transformation then $gf$ stands for the composition $g\circ f$. Hence $(gf)\omega=g(f\omega)$.

Let $S$ be a vector semigroup. By its \emph{action} on a measurable space $\Omega$ we mean a family of transformations $(f^x\mid x\in S)$ with the following two properties: (1) $f^{x+y}=f^xf^y$ for all $x,y\in S$, and (in case $0\in S$) (2) $f^0$ is the identity on $\Omega$. If $0\not\in S$ and $(f^x\mid x\in S)$ is an action of $S$ on $\Omega$, we can extend it to the action $(f^x\mid x\in S\cup\{0\})$ by adding the identity transformation as $f^0$.

If $(f^x\mid x\in S)$ is an action then $x,-x\in S$ implies that $f^x$ is a bijection between $\Omega$ and $\Omega$, and $f^{-x}$ is its inverse. For any $x\in S$ and $W\subset\Omega$, we denote $f^{-x}(W)=(f^x)^{-1}(W)$. If $-x\in S$, that set coincide with the image of $W$ by the transformation $f^{-x}$, therefore our notation is not ambiguous. We say that a measurable set $W$ is \emph{$f^x$-invariant} if $f^{-x}(W)=W$.

Let $f=(f^x\mid x\in S)$ be an action of $S$ on $\Omega$ and $\Prob$ some probability on $\Omega$. We say that $\Prob$ is \emph{$f^x$-invariant} if $\Prob(f^{-x}(W))=\Prob(W)$ for measurable $W\subset\Omega$. If $\Prob$ is $f^x$-invariant for all $x\in S$, we call $f$ an \emph{action of $S$ on the probability space $(\Omega,\Prob)$}. We call that action \emph{ergodic}, if $\Prob(W)\in\{0,1\}$ for all measurable $W\subset\Omega$, which is $f^x$-invariant for some $0\ne x\in S$.

The proof of any ergodic theorem is usually preceded by some "maximal ergodic theorem". In our case the role of it is taken by the following theorem, although its formulation does not contain any $\max$.

\begin{theorem}\label{t:maxerg}
Let $k\ge 1$, $S$ be a $k$-dimensional vector semigroup, $A\subset S$ a finite set, $C=\cone(A)$, $S_C=S\cap C$ and $0\ne a\in A$. Let $(f^x\mid x\in S)$ be an ergodic action on a probability space $(\Omega,\Prob)$ and $Z\in L^k(\Prob)$ a random variable, such that $\Mean Z=0$ and the family $(Z(f^x)\mid x\in S)$ is almost independent. Then almost surely
\begin{equation*}
  \frac{1}{\abs{x}+s}\sum_{i=1}^{s}Z(f^{x+(i-1)a})\to 0,\quad\text{as $S_C\times\Zd_+\ni (x,s)\to\infty$.}
\end{equation*}
\end{theorem}

\begin{proof}
For short, set
\begin{equation*}
  \bar Z(x,s,\omega)=\sum_{i=1}^sZ(f^{x+(i-1)a}\omega).
\end{equation*}
It suffices to prove that almost surely
\begin{equation}\label{e:limsup}
  \varlimsup_{S_C\times\Zd_+\ni(x,s)\to\infty}\frac{\bar Z(x,s)}{\abs{x}+s}\le 0,
\end{equation}
because then the analogous inequality with $-Z$ instead of $Z$ gives
\begin{equation*}
  \varliminf_{S_C\times\Zd_+\ni(x,s)\to\infty}\frac{\bar Z(x,s)}{\abs{x}+s}\ge 0.
\end{equation*}

Denote $A_0=\{x\in A\mid -x\in S_C\}$, $A_1=A\be A_0$, $L_0=\lin(A_0)$ and $L=\lin(A)$. The cases, where $a\in A_0$ and where $a\in A_1$, are a bit different, but the difference is not very big, so we consider them together. Let $\calB$ be the set of all $B\subset A_1$, such that the family $(x+L_0\mid x\in B)$ is a basis of the quotient space $L/L_0$. By Proposition~\ref{t:d}, $S_C$ is the union of finitely many sets of the form $b+\gp(A_0)+\sg(B)$, where $b\in S_C$, $B\in\calB$ and (in the case, where $a\in A_1$) $a\in B$. Therefore it suffices to prove \refeq{limsup} with $b+S_0$ instead of $S_C$, where $S_0=\gp(A_0)+\sg(B)$.

Obviously, $\bar Z(b+y,s,\omega)=\bar Z(y,s,f^b\omega)$ and
\begin{equation*}
  \abs{b+y_n}+s_n\to\infty\iff\abs{y_n}+s_n\to\infty.
\end{equation*}
Since $\Prob$ is $f^b$-invariant, it suffices to prove that almost surely
\begin{equation}\label{e:limsup2}
  \varlimsup_{S_0\times\Zd_+\ni (y,s)\to\infty}\frac{\bar Z(y,s)}{\abs{y}+s}\le 0.
\end{equation}
Let $(a_1,\dots,a_p)\subset A_0$ be a basis of $L_0$ and $B=\{a_{p+1},\dots,a_k\}$. If $a\in A_0$, we may assume that $a_1=a$, and if $a\in A_1$ then $a\in B$ and we assume that $a_k=a$. Let $L'=\lin(a_2,\dots,a_k)$ in the first case and $L'=(a_1,\dots,a_{k-1})$ in the second, then in both cases $L=\Rd a\oplus L'$. Find a $c<\infty$, such that, for all $t\in\Rd$ and $z\in L'$,
\begin{equation*}
  \abs{t}+\abs{z}\le c\abs{ta+z}.
\end{equation*}

Each $y\in S_0$ is written in a unique way as $y=ta+z$ with $t\in\Zd$ and $z\in L'\cap S_C$; moreover, $t\ge 0$ in the case, where $a\in A_1$. The following identities then are easily checked: if $t\ge 0$ then
\begin{equation}\label{e:Z1}
  \bar Z(y,s,\omega)
  =\sum_{i=1}^{t+s}Z(f^{z+(i-1)a}\omega)
    -\sum_{i=1}^{t}Z(f^{z+(i-1)a}\omega),
\end{equation}
if $t+s\le 0$ then
\begin{equation}\label{e:Z2}
  \bar Z(y,s,\omega)
  =\sum_{j=1}^{\abs{t}+1}Z(f^{z-(j-1)a}\omega)
    -\sum_{j=1}^{\abs{t}+1-s}Z(f^{z-(j-1)a}\omega),
\end{equation}
and if $t<0<t+s$ then
\begin{equation}\label{e:Z3}
  \bar Z(y,s,\omega)
  =-Z(f^z\omega)+\sum_{j=1}^{\abs{t}+1}Z(f^{z-(j-1)a}\omega)+\sum_{i=1}^{-\abs{t}+s}Z(f^{z+(i-1)a}\omega).
\end{equation}

Consider the case, where $k=1$. Let $W$ be the set of all outcomes $\omega$, such that, for all $x\in S$,
\begin{equation*}
  \frac1n\sum_{i=1}^nZ(f^{(i-1)x}\omega)\tends{n\to\infty}0.
\end{equation*}
By the ergodic theorem, $\Prob(W)=1$. Let us fix $\omega\in W$ and prove that  \refeq{limsup2} holds at $\omega$. Let $S_0\times\Zd_+\ni(y_n,s_n)\to\infty$ be any sequence, for which
\begin{equation}\label{e:limsup3}
  \varlimsup_{S_0\times\Zd_+\ni (y,s)\to\infty}\frac{\bar Z(y,s,\omega)}{\abs{y}+s}
  =\lim_{n\to\infty}\frac{\bar Z(y_n,s_n,\omega)}{\abs{y_n}+s_n}.
\end{equation}
We need to show that the limit on the right hand side is nonpositive. In the considered case $L'=\{0\}$, therefore $y_n=t_na$ for some $t_n\in\Zd$, and $\abs{t_n}+s_n\to\infty$. Without loss of generality we can assume that either $t_n\ge 0$ for all $n$, or $t_n+s_n\le 0$ for all $n$, or $t_n<0<s_n$ for all $n$.

If $t_n\ge 0$ for all $n$ then without loss of generality we can assume that either $t_n\to\infty$ or $t_n=O(1)$. In both cases \refeq{Z1} implies
\begin{equation*}
  \bar Z(y_n,s_n,\omega)
  =\sum_{i=1}^{t_n+s_n}Z(f^{(i-1)a}\omega)
    -\sum_{i=1}^{t_n}Z(f^{(i-1)a}\omega)
  =o(t_n+s_n)
  =o(\abs{y_n}+s_n).
\end{equation*}
If $t_n+s_n\le 0$ for all $n$ then $0\le s_n\le\abs{t_n}$ and therefore $\abs{t_n}\to\infty$. Without loss of generality we can assume that either $\abs{t_n}-s_n\to\infty$ or $\abs{t_n}-s_n=O(1)$. In both cases \refeq{Z2} yields
\begin{multline*}
  \bar Z(y_n,s_n,\omega)
  =\sum_{j=1}^{\abs{t_n}+1}Z(f^{-(j-1)a}\omega)
    -\sum_{j=1}^{\abs{t_n}+1-s_n}Z(f^{-(j-1)a}\omega)\\
  =o(\abs{t_n}+s_n)
  =o(\abs{y_n}+s_n).
\end{multline*}
If $t_n<0<t_n+s_n$ for all $n$, then $0\le \abs{t_n}\le s_n$ and therefore $s_n\to\infty$. Without loss of generality we can assume that either $s_n-\abs{t_n}\to\infty$ or $s_n-\abs{t_n}=O(1)$, and that either $\abs{t_n}\to\infty$, or $\abs{t_n}=O(1)$. In all 4 cases \refeq{Z3} implies
\begin{multline*}
  \bar Z(y_n,s_n,\omega)
  =-Z(\omega)+\sum_{j=1}^{\abs{t_n}+1}Z(f^{-(j-1)a}\omega)+\sum_{i=1}^{s_n-\abs{t_n}}Z(f^{(i-1)a}\omega)\\
  =o(\abs{t_n}+s_n)
  =o(\abs{y_n}+s_n).
\end{multline*}

Now consider the case, where $k\ge 2$. For any $x\in S$ and $\epsilon$ denote
\begin{gather*}
  M_+(x,\epsilon)=\sup_{s\ge 0}\sum_{i=1}^s(Z(f^{(i-1)x})-\epsilon),\quad
  M_+(x,\epsilon)=\sup_{s\ge 0}\sum_{i=1}^s(-Z(f^{(i-1)x})-\epsilon)
\end{gather*}
and
\begin{equation*}
  M(x,\epsilon)=M_+(x,\epsilon)+M_-(x,\epsilon).
\end{equation*}
By the almost independence assumption and Proposition~\ref{t:M}, $\Mean M(x,\epsilon)^{k-1}<\infty$. The intersection $L'\cap S_C$ is a vector semigroup and its dimension does not exceed $\dim{L'}=k-1$. Therefore, by Proposition~\ref{t:max}, for almost all $\omega$,
\begin{equation}\label{e:Mz}
  \frac{M(x,\epsilon,f^z\omega)}{\abs{z}}\to 0,\quad\text{kai $L'\cap S_C\ni z\to\infty$}.
\end{equation}

Let us look at how random variables $\bar Z(y,s)$ with $y=ta+z$ are dominated by $M(\pm a,\epsilon)$. We will assume that $s\ge 1$, but the obtained inequalities will be also valid for $s=0$, because $\bar Z(y,0)=0$. If $t\ge 0$ then \refeq{Z1} implies
\begin{align*}
  \bar Z(y,s,\omega)
  &=\sum_{i=1}^{t+s}(Z(f^{z+(i-1)a}\omega)-\epsilon)
    +\sum_{i=1}^{t}(-Z(f^{z+(i-1)a}\omega)-\epsilon)+\epsilon(2t+s)\\
  &\le M(a,\epsilon,f^z\omega)+\epsilon(2c\abs{y}+s).
\end{align*}
If $t+s\le 0$ then \refeq{Z2} yields analogously
\begin{align*}
  &\bar Z(y,s,\omega)\\
  &=\sum_{j=1}^{\abs{t}+1}(Z(f^{z-(j-1)a}\omega)-\epsilon)
    +\sum_{j=1}^{\abs{t}+1-s}(-Z(f^{z-(j-1)a}\omega)-\epsilon)+\epsilon(2\abs{t}+2-s)\\
  &\le M(-a,\epsilon,f^z\omega)+\epsilon(2c\abs{y}+s).
\end{align*}
Finally, if $t<0<t+s$ then by \refeq{Z3},
\begin{align*}
  \bar Z(y,s,\omega)
  &=\sum_{j=1}^{\abs{t}}(Z(f^{z-(j-1)a}\omega')-\epsilon)+\sum_{i=1}^{-\abs{t}+s}(Z(f^{z+(i-1)a}\omega)-\epsilon)+\epsilon s\\
  &\le M(-a,\epsilon,f^{z}\omega')+M(a,\epsilon,f^{z}\omega)+\epsilon s
\end{align*}
with $\omega'=f^{-a}\omega$.

Let $\tilde W$ be the set off all outcomes $\omega$, such that \refeq{Mz} holds for $x=a$ and for $x=-a$ (in the case, where $a\in A_0$), and for all $\eps=\eps_l$, where $(\eps_l)$ is some fixed sequence tending to 0. Denote $W_1=\tilde W$ if $a\in A_1$, and $W_1=\tilde W\cap f^a(\tilde W)$ if $a\in A_0$. In both cases $\Prob(W_1)=1$. Let $W_2$ be the set of all outcomes $\omega$, for which
\begin{equation*}
  \frac1s\sum_{i=1}^sZ(f^{x+(i-1)a}\omega)\tends{s\to\infty}0,\quad x\in S.
\end{equation*}
By the ergodic theorem, $\Prob(W_2)=1$. Set $W=W_1\cap W_2$, then also $\Prob(W)=1$.

Fix $\omega\in W$ and prove that \refeq{limsup2} holds at $\omega$. Let $S_0\times\Zd_+\ni(y_n,s_n)\to\infty$ be any sequence, for which \refeq{limsup3} holds. Without loss of generality we can assume that either $\abs{y_n}\to\infty$, or $s_n\to\infty$ and all $y_n$ coincide with some fixed $y\in S_0$. In the second case limit \refeq{limsup3} equals 0, because $\omega\in W_2$. It remains to consider the first case --- where $\abs{y_n}\to\infty$.

Let $y_n=t_na+z_n$ with $t_n\in\Zd$ and $z_n\in L'\cap S_C$ (in the case, where $a\in A_1$, all $t_n$ are nonnegative). If $a\in A_1$, it follows from the majorization inequalities obtained above that
\begin{equation*}
  \frac{\bar Z(y_n,s_n,\omega)}{\abs{y_n}+s_n}\le\frac{M(a,\epsilon,f^{z_n}\omega)}{\abs{y_n}}+\epsilon(2c+1).
\end{equation*}
Without loss of generality we can assume that either $\abs{z_n}\to\infty$, or $t_n\to\infty$ and all $z_n$ coincide with some fixed $z\in L'\cap S_c$. In both cases the first summand in the right hand side of the inequality tends to 0, for any $\epsilon=\epsilon_l$ (in the first case this is true, because $\omega\in\tilde W$ and $\abs{y_n}\ge \abs{z_n}/c$, and in the second case --- because $\abs{y_n}\to\infty$). Hence
\begin{equation*}
  \lim_{n\to\infty}\frac{\bar Z(y_n,s_n,\omega)}{\abs{y_n}+s_n}\le\epsilon_l(2c+1)
\end{equation*}
for all $l$, which means that limit \refeq{limsup3} is indeed non-positive.

If $a\in A_0$ then the majorization inequalities yield
\begin{equation*}
  \frac{\bar Z(y_n,s_n,\omega)}{\abs{y_n}+s_n}\le\frac{M(a,\epsilon,f^{z_n}\omega)+M(-a,\epsilon,f^{z_n}\omega)+M(-a,\epsilon,f^{z_n}\omega')}{\abs{y_n}}+\epsilon(2c+1),
\end{equation*}
where $\omega'=f^a\omega$. For $\epsilon=\epsilon_l$, the first summand in the right hand side tends to 0 in both cases, where $\abs{z_n}\to\infty$ (because $\omega,\omega'\in\tilde W$), and where $z_n=z$ for all $n$. Hence limit \refeq{limsup3} is again non-positive.
\end{proof}

\paragraph{Subadditive random functions.}

Let $f=(f^x\mid x\in S)$ be an action of a vector semigroup $S$ on a probability space $(\Omega,\Prob)$ and $h=(h(x)\mid x\in S)$ a family of random variables defined on  $\Omega$. We think of $h$ as of a random function from $S$ to $\Rd$ and denote the value of $h(x)$ at outcome $\omega\in\Omega$ by $h(x,\omega)$. A random function $h$ is called \emph{$f$-subadditive}, if $h(0)=0$ and, for all $x,y\in S$ and $\omega\in\Omega$,
\begin{equation*}
  h(x+y,\omega)\le h(x,\omega)+h(y,f^x\omega).
\end{equation*}

If $h$ is $f$-subadditive and $\Mean h(x)^+<\infty$ for all $x\in S$ then, for all $x\in S$, the sequence $(h(nx)\mid n\ge 1)$ is subadditive in the usual sense and, by the Kingman's subadditive ergodic theorem, $h(nx)/n$ tends almost surely to some random variable $q(x)$. If the action $f$ is ergodic, the limit random variable is degenerate, therefore we can assume that $q$ is a non random function on $S$. Then, by the same Kingman's theorem, $q(x)=\lim_{n\to\infty}\Mean h(nx)/n$. By subadditivity and invariance, for all $x,y\in S$ and $t\in\Nd$,
\begin{equation*}
  q(x+y)=\lim_{n\to\infty}\frac{\Mean h(nx+ny)}{n}\le\lim_{n\to\infty}\frac{\Mean h(nx)+\Mean h(ny)}{n}=q(x)+q(y),
\end{equation*}
and
\begin{equation*}
  q(tx)=\lim_{n\to\infty}\frac{\Mean h(ntx)}{n}=t\lim_{n\to\infty}\frac{\Mean h(ntx)}{nt}=tq(x).
\end{equation*}
Hence $q$ is a $\Zd$-gauge on the semigroup $S$. By Proposition~\ref{t:kalibras}, the restriction of $q$ on $O\cap S$ (where $O$ is the asymptotic cone of $S$) is extended in the unique way to a gauge on $O$, which is called \emph{the gauge associated with $h$}.

\paragraph{Proof of Theorem~\ref{t:main}.}

Let $S$ be $k$-dimensional. For each $S$-cone $C=\cone(a_1,\dots,a_k)$ with all $\pi(a_i)\ne 0$, and each $j=1,\dots,k$, let $\tilde W(C,j)$ denote the set of all outcomes $\omega$, for which
\begin{equation*}
  \frac{1}{\abs{y}+s}\sum_{i=1}^s(h_+(a_j,g^{y+(i-1)\pi(a_j)}\omega)-\Mean h_+(a_j))\to 0,
\end{equation*}
as $\pi(S\cap C)\times\Zd_+\ni (y,s)\to\infty$. By Theorem~\ref{t:maxerg}, $\Prob(\tilde W(C,j))=1$. Set $\tilde W=\bigcap_{C,j}\tilde W(C,j)$ and $W_1=\bigcap_{y\in T}g^{-y}(\tilde W)$. Since the set of all $S$-cones is countable, $\Prob(W_1)=1$.

Let $W_2$ and $W_3$ be the sets of all outcomes $\omega$, such that, for all $a\in S$, respectively,
\begin{equation*}
  \frac{h_+(a,g^y\omega)}{\abs{y}}\to 0,\quad\text{as $T\ni y\to\infty$,}
\end{equation*}
and
\begin{equation*}
  \frac{h(na,\omega)}{n}\tends{n\to\infty}q(a).
\end{equation*}
By Proposition~\ref{t:max} and the Kingman's subadditive ergodic theorem, $\Prob(W_2)=\Prob(W_3)=1$. Set $W=W_1\cap W_2\cap W_3$, then $\Prob(W)=1$ as well.

Let us prove that \refeq{shape} holds for any $\omega\in W$ and any sequence $S\ni x_n\to\infty$ with $x_n/\abs{x_n}\to x\in O$. By Proposition~\ref{t:vidiniskugis}, there exists an $S$-cone $C=\cone(a_1,\dots,a_k)$, such that $\pi(a_i)\ne 0$ for all $i$ and $x\in\rint(C)$. We can assume that $x_n\in C$ for all $n$. Denote $A=\{a_1,\dots,a_k\}$ and $S_C=S\cap C$. Since the cone $C$ is pointed, Proposition~\ref{t:d} yields that $S_C=A_0+\sg(A)$ with some finite set $A_0\subset S_C$. The sequence $h(x_n,\omega)/\abs{x_n}$ is thus decomposed into a finite number of subsequences (in each subsequence $x_n=a+x_n'$ with some $a\in A_0$ and $x_n'\in\sg(A)$) and it suffices to prove that each subsequence tends to $q(x)$.

Suppose $x_n=a+x_{1,n}$ with $a\in A_0$ and $x_{1,n}\in\sg(A)$. By subadditivity,
\begin{equation*}
  \frac{h(x_n,\omega)}{\abs{x_n}}
  \le\frac{h(x_{1,n},\omega)}{\abs{x_n}}+\frac{h_+(a,g^{\pi(x_{1,n})}\omega)}{\abs{x_n}}.
\end{equation*}
The second summand in the right hand side tends to 0, because so does any its subsequence with $\abs{\pi(x_{1,n})}\to\infty$ (since $\omega\in W_2$ and $\abs{\pi(x_{1,n})}\le \norm{\pi}\,\abs{x_{1,n}}\sim \norm{\pi}\,\abs{x_n}$), and also any its subsequence with $\pi(x_{1,n})=y$ (since $\abs{x_n}\to\infty$). Therefore
\begin{equation*}
  \frac{h(x_n,\omega)}{\abs{x_n}}
  \le\frac{h(x_{1,n},\omega)}{\abs{x_n}}+o(1).
\end{equation*}

By Proposition~\ref{t:sg}, there exists a $d\in\Nd$, such that $dz\in\sg(A)$ for all $z\in S_C$. Denote $a'=(d-1)a$ and $x_{2,n}=a'+x_n=da+x_{1,n}$. Then $a'\in S_C$, $x_{2,n}\in\sg(A)$ and similarly to above we get
\begin{equation*}
  \frac{h(x_n,\omega)}{\abs{x_n}}
  \ge\frac{h(x_{2,n},\omega)}{\abs{x_n}}-\frac{h_+(a',g^{\pi(x_{2,n})}\omega)}{\abs{x_n}}
  =\frac{h(x_{2,n},\omega)}{\abs{x_n}}+o(1).
\end{equation*}
Since $\abs{x_{j,n}}\sim\abs{x_n}$, it suffices to prove that for $j=1,2$
\begin{equation*}
  \frac{h(x_{j,n},\omega)}{\abs{x_{j,n}}}\to q(x).
\end{equation*}
We can use the facts that $x_{j,n}\in\sg(A)$ and $x_{j,n}/\abs{x_{j,n}}\to x$.

We do not need anymore the sequence $(x_n)$ used to build $x_{j,n}$, so we omit, for short, the index $j$ and write $x_n$ instead of $x_{j,n}$. Let $z^i$ denote the coordinates of a vector $z\in L$ in the basis  $(a_1,\dots,a_k)$.

Fix $\epsilon<\frac12$ and find $y\in \rint(C)\cap S^*$, such that
\begin{equation}\label{e:xy}
  1-\epsilon<\frac{x^i}{y^i}<1+\epsilon\quad\text{for $i=1,\dots,k$.}
\end{equation}
Let $p$ be a natural number, such that $py^i\in\Nd$ for all $i$. Define
\begin{equation*}
  s_n=\min_i\Floor{\frac{(1-\epsilon)x_n^i}{py^i}}p,\quad
  t_n=\max_i\Ceiling{\frac{(1+\epsilon)x_n^i}{py^i}}p.
\end{equation*}
For all $i$,
\begin{equation*}
  s_ny^i\le x_n^i\le t_ny^i
\end{equation*}
and all three numbers are integers, therefore $x_n-s_ny\in \sg(A)$ and $t_ny-x_n\in \sg(A)$.

Clearly, $x_n^i\sim x^i\abs{x_n}$ and therefore $s_n\sim\abs{x_n}s$, $t_n\sim\abs{x_n}t$ with
\begin{equation}\label{e:st}
  s=(1-\epsilon)\min_i\frac{x^i}{y^i}\quad\text{ir}\quad
  t=(1+\epsilon)\max_i\frac{x^i}{y^i}.
\end{equation}
Then $(1-\epsilon)^2<s\le t<(1+\epsilon)^2$, hence, for $n$ large enough,
\begin{equation*}
  (1-\epsilon)x^i\abs{x_n}<x_n^i<(1+\epsilon)x^i\abs{x_n}\quad\text{and}\quad
  (1-\epsilon)^2\abs{x_n}<s_n\le t_n<(1+\epsilon)^2\abs{x_n}.
\end{equation*}
It yields
\begin{equation*}
  x_n^i-s_ny^i
  \le(1+\epsilon)x^i\abs{x_n}-(1-\epsilon)^2y^i\abs{x_n}
  \le\Bigl(1+\epsilon-\frac{(1-\epsilon)^2}{1+\epsilon}\Bigr)x^i\abs{x_n}
  \le 4\epsilon x^i\abs{x_n}
\end{equation*}
and
\begin{equation*}
  t_ny^i-x_n^i
  \le(1+\epsilon)^2y^i\abs{x_n}-(1-\epsilon)x^i\abs{x_n}
  \le\Bigl(\frac{(1+\epsilon)^2}{1-\epsilon}-1+\epsilon\Bigr)x^i\abs{x_n}
  \le 8\epsilon x^i\abs{x_n}.
\end{equation*}
Moreover,
\begin{equation*}
  \frac{x_n^i-s_ny^i}{\abs{x_n}}
  \to x^i-sy^i
  \ge x^i-(1-\epsilon)x^i=\epsilon x^i
\end{equation*}
and
\begin{equation*}
  \frac{t_ny^i-x_n^i}{\abs{x_n}}
  \to ty^i-x^i
  \ge (1+\epsilon)x^i-x^i=\epsilon x^i,
\end{equation*}
hence $x_n^i-s_ny^i\to\infty$ and $t_ny^i-x_n^i\to\infty$.

Now, by subadditivity,
\begin{equation}\label{e:jaubeveik1}
  h(x_n,\omega)\le h(s_ny,\omega)+h(x_n-s_ny,f^{s_ny}\omega)
\end{equation}
and
\begin{equation}\label{e:jaubeveik2}
  h(x_n,\omega)\ge h(t_ny,\omega)-h(t_ny-x_n,f^{x_n}\omega);
\end{equation}
moreover,
\begin{equation}\label{e:jaubeveik3}
  \frac{h(s_ny,\omega)}{\abs{x_n}}
  \to s q(y)\quad\text{ir}\quad
  \frac{h(t_ny,\omega)}{\abs{x_n}}\to tq(y),
\end{equation}
because $\omega\in W_3$.

Let us estimate two remaining terms. Denote $z_n=x_n-s_ny$, then $z_n^i\to\infty$ and $z_n^i\le 4\epsilon x^i\abs{x_n}$ for $n$ large enough. Again by subadditivity,
\begin{equation*}
  h(z_n,f^{s_ny}\omega)
  \le\sum_{i=1}^k\sum_{j=1}^{z_n^i}h_+(a_i,f^{y_{in}+(j-1)a_i}\omega),
\end{equation*}
where $y_{in}=s_ny+\sum_{1\le i'<i}z_n^{i'}a_{i'}$. Since $z_n^i\to\infty$ and $\omega\in W_1$, we get
\begin{multline*}
  \sum_{j=1}^{z_n^i}(h_+(a_i,f^{y_{in}+(j-1)a_i}\omega)-\Mean h_+(a_i))
  =o(z_n^i+\abs{\pi(y_{in})})\\
  =o(z_n^i+\norm{\pi}\,\abs{y_{in}})
  =o(\abs{x_n}).
\end{multline*}
Therefore
\begin{equation*}
  \varlimsup_{n\to\infty}\frac{h(x_n-s_ny,f^{s_ny}\omega)}{\abs{x_n}}
  \le\sum_{i=1}^k\varlimsup_{n\to\infty}\frac{z_n^i}{\abs{x_n}}\Mean h_+(a_i)
  \le c\epsilon
\end{equation*}
with $c=4\sum_{i=1}^kx^i\Mean h_+(a_i)$.

Similarly to above,
\begin{equation*}
  \varlimsup_{n\to\infty}\frac{h(t_ny-x_n,f^{x_n}\omega)}{\abs{x_n}}\le 2c\epsilon.
\end{equation*}
Then \refeq{jaubeveik1}--\refeq{jaubeveik3} yield
\begin{equation*}
  tq(y)-2c\epsilon\le\varliminf_{n\to\infty}\frac{h(x_n,\omega)}{\abs{x_n}}\le \varlimsup_{n\to\infty}\frac{h(x_n,\omega)}{\abs{x_n}}\le sq(y)+c\epsilon.
\end{equation*}
Here $\epsilon$ can be arbitrary small, $y$ depends on $\epsilon$ but satisfies \refeq{xy}, $s$ and $t$ are calculated by \refeq{st}, and $c$ does not depend on $\epsilon$. If $\epsilon$ approaches 0, $y$ tends to $x$, and $s,t$ to 1. Therefore taking the limits, as $\epsilon\to 0$, we get $h(x_n,\omega)/\abs{x_n}\to q(x)$.

%\bibliography{gen}

\end{document}